\documentclass[10pt]{article}

\usepackage{latexsym, amssymb, amsmath, amsthm, amscd}
\usepackage[all,cmtip]{xy}
\usepackage{amsmath}
\usepackage{amsthm}
\usepackage{amssymb}
\usepackage{relsize}
\usepackage{mathtools}
\usepackage{amsfonts}
\usepackage{amsrefs}

\usepackage{color,authblk}
\usepackage{tikz}
\usepackage{amscd} 
\usepackage{comment}
\usepackage{mathrsfs}
\usepackage[all]{xy}
\usepackage{stackrel}
\usepackage{graphicx}
\usepackage{authblk}

\usepackage[hidelinks]{hyperref}
\usepackage{color,soul}
\usepackage{fullpage}
\usepackage[all,cmtip]{xy}
\usepackage{amsmath,amscd}
\usepackage{tikz-cd}
\usepackage{tikz}

\DeclareMathAlphabet{\matholdcal}{OMS}{cmsy}{m}{n}

\usetikzlibrary{matrix}

\usepackage[mathcal]{euscript}
\usepackage{txfonts}

\usepackage{hyperref}

\numberwithin{equation}{section} \DeclareMathSizes{2}{10}{12}{13}
\newtheorem{thm}{Proposition}[section]
\newtheorem{Thm}[thm]{Theorem}

\newtheorem{cor}[thm]{Corollary}
\newtheorem{lem}[thm]{Lemma}
\newtheorem{defn}[thm]{Definition}

\numberwithin{thm}{section} 

\title{Entwined comodules and contramodules over coalgebras with several objects: Frobenius, separability and Maschke theorems}

\author{Abhishek Banerjee \footnote{Department of Mathematics, Indian Institute of Science, Bangalore, India. Email: abhishekbanerjee1313@gmail.com} $\qquad\qquad$ Surjeet Kour \footnote{Department of Mathematics, Indian Institute of Technology, Delhi, India. Email: koursurjeet@gmail.com}}

\date{}

\begin{document}

\maketitle 

\medskip

\begin{abstract} We study module like objects over categorical quotients of algebras by the action of  coalgebras with several objects. These take the form of ``entwined comodules'' and ``entwined contramodules'' over a triple $(\mathscr C,A,\psi)$, where $A$ is an algebra, $\mathscr C$ is a coalgebra with several objects and $\psi$ is a collection of maps that ``entwines'' $\mathscr C$ with $A$. Our objective is to prove Frobenius, separability and Maschke type theorems
for functors between categories of entwined comodules and entwined contramodules. 
\end{abstract}

\medskip

{\bf MSC(2020) Subject Classification:} 16T15, 18E10
\medskip

\medskip

\section{Introduction}

An entwining structure $(C,A,\psi)$ consists of a coalgebra $C$, an algebra $A$ and a map $\psi:C\otimes A\longrightarrow A\otimes C$ that weaves together the comultiplication and multiplication structures in a manner  similar to a bialgebra. This notion, introduced by Brzezi\'{n}ski and Majid \cite{BrMj}, is extremely versatile. For instance, an entwining structure captures the essence of a coalgebra Galois extension, or may be seen as a noncommutative replacement for a principal fiber bundle produced by the quotient of the free action of an  affine algebraic group on a scheme. The properties of such a noncommutative space are then  understood in terms of the category of modules over it. In general, it was discovered that modules over an entwining structure (introduced in \cite{Brz1999}) bring together several objects of study in the literature, such as relative Hopf modules, Doi-Hopf modules and Yetter-Drinfeld modules (see for instance, \cite{BBR-}, \cite{BBR1}, \cite{AB}, \cite{uni}, \cite{Brz1}, \cite{Brz2002}, \cite{Bul1}, \cite{Bul2}, \cite{CaDe}, \cite{SchP}). In this paper, we prove Frobenius, separability and Maschke type theorems for entwined comodules and entwined contramodules for coalgebras with several objects.

\smallskip
  For a field $K$,   the  general philosophy of Mitchell \cite{Mit} is that  a small $K$-linear category should be seen as a ``$K$-algebra with several objects.''  Hence, most results on modules over a $K$-algebra can be extended to modules over small $K$-linear categories. We continue with the counterpart of this intuition for coalgebras. A $K$-coalgebra $\mathscr C$ with several objects (see Day and Street \cite{DS}, McCrudden \cite{Mc1}) is a small category enriched over the opposite of the category of $K$-vector spaces. Accordingly, we have $K$-linear maps
\begin{equation}\label{11c2.1}
\delta_{XYZ}:\mathscr C(X,Z)\longrightarrow \mathscr C(Y,Z)\otimes \mathscr C(X,Y) \qquad \epsilon_X: \mathscr C(X,X)\longrightarrow K
\end{equation} for every $X$, $Y$, $Z\in Ob(\mathscr C)$ that satisfy the required coassociativity and counit conditions. By an entwining of $\mathscr C$ with a $K$-algebra $A$, we will mean a collection of $K$-linear maps
\begin{equation}\label{112.2e}
\psi=\{\psi_{XY}:\mathscr C(X,Y)\otimes A\longrightarrow A\otimes \mathscr C(X,Y) \}_{X,Y\in Ob(\mathscr C)}
\end{equation} sarisfying certain conditions that we describe in Section 2. The  structure $(\mathscr C,A,\psi)$ may be seen as  a categorical quotient of the noncommutative space produced by the action of $\mathscr C$ on $A$.  This noncommutative space $(\mathscr C,A,\psi)$ need not exist as an explicit geometric object, but it can be studied by means of module like objects over it. 

\smallskip 
The structure  $(\mathscr C,A,\psi)$ has two ``module like categories'' canonically associated to it: entwined comodules and entwined contramodules. An entwined comodule consists of a comodule object over $\mathscr C$ equipped with $A$-linear structures that are well behaved with respect to the entwining $\psi$. Similarly, an entwined contramodule consists of a contramodule object over $\mathscr 
C$  equipped with $A$-linear structures that are well behaved with respect to the entwining $\psi$. We mention here that one of our aims is to treat the theory of contramodules at par with that of comodules.  We know that the structure map of an $A$-module $M$ can be described in two equivalent ways, either as a linear map $M\otimes A\longrightarrow M$ or as a linear map
$M\longrightarrow Hom_K(A,M)$. Over a coalgebra $C$, the notion of module can therefore be dualized in two separate ways, with structure maps given by
\begin{equation}\label{1.3dx}
N\longrightarrow N\otimes C\qquad\mbox{or}\qquad Hom_K(C,N)\longrightarrow N
\end{equation} The first of these structure maps in \eqref{1.3dx} leads to comomodules over $C$, while the second leads to  contramodules over $C$. While the notion of contramodules is classical, introduced by Eilenberg and Moore \cite[$\S$ IV.5]{EiMo}, the development of the theory of contramodules has somewhat lagged behind that of the theory of comodules. However, in recent years, there has been a lot of interest in the subject of contramodules (see \cite{BBR2}, \cite{Baz}, \cite{P}, \cite{Pos}, \cite{Pos1}, \cite{Pos0}, \cite{Shap}, \cite{Wiscc}). 

\smallskip
In classical ring theory, the properties of a ring are often recovered from its category of modules. Similarly, the properties of a morphism of rings are obtained from the adjoint pair consisting of extension and restriction of scalars. While the noncommutative quotient space corresponding to $(\mathscr C,A,\psi)$ need not exist as an explicit geometric object, there are adjoint functors between categories of entwined comodules or entwined contramodules that behave like extension and restriction of scalars. We study the properties of such adjoint pairs to give conditions under which   morphisms between these  noncommutative spaces behave like Frobenius or separable extensions of rings. In the case of entwined comodules, we adapt the methods of Brzezi\'{n}ski, Caenepeel, Militaru and Zhu \cite{uni} used to obtain Frobenius and separability theorems for modules over an entwining of a usual coalgebra with an algebra. The unified approach with entwined modules  in \cite{uni} brings together several ideas that had been used in   \cite{CMZ}, \cite{C13}, \cite{C10}, \cite{C11} to prove a range of Frobenius, separability and Maschke type theorems for objects such as Doi-Hopf modules, Yetter-Drinfeld modules 
and relative Hopf modules.  In \cite{BBR1}, \cite{AB} we have used the  methods of \cite{uni} to study  Frobenius and separability conditions for  modules over small preadditive categories entwined with coalgebras. In \cite{AB}, we considered quiver representations of entwining structures inspired by Estrada and Virili \cite{EV}, giving rise to scheme like objects, along with ``sheaves of modules'' and ``quasi-coherent sheaves'' over them.
 
\smallskip
We now describe the paper in more detail. An entwined comodule $\mathcal M$ over an entwining structure $(\mathscr C,A,\psi)$ consists of vector spaces 
$\{\mathcal M(X)\}_{X\in Ob(\mathscr C)}$ and is equipped with coaction maps as well as action maps 
\begin{equation}
\mathcal M(X)\longrightarrow \mathcal M(Y)\otimes \mathscr C(X,Y) \qquad \mathcal M(X)\otimes A\longrightarrow \mathcal M(X)\qquad X,Y\in Ob(\mathscr C)
\end{equation} satisfying compatibility conditions that we have laid out in Definition \ref{D2.1}.  We remark here that since the structure of $\mathcal M$ consists of both actions and coactions, it would be just as fair to refer to it  as an entwined module over $(\mathscr C,A,\psi)$. However since we want to work with both comodules and contramodules over the coalgebra $\mathscr C$ with several objects, we prefer the term entwined comodule. 

\smallskip We let $Com^{\mathscr C}_A(\psi)$ denote the category of entwined comodules over $(\mathscr C,A,\psi)$. Our first main result (see Theorem \ref{T2.5}) is that $Com^{\mathscr C}_A(\psi)$ is a Grothendieck category. In Section 2, we also construct a pair of adjoint functors 
\begin{equation}\label{adjsec1}
F_\psi :Com^{\mathscr C}\longrightarrow Com_A^{\mathscr C}(\psi) \qquad G_\psi :Com_A^{\mathscr C}(\psi)\longrightarrow Com^{\mathscr C}
\end{equation} where $Com^{\mathscr C}$ denotes the category of comodules over the coalgebra $\mathscr C$ with several objects. As mentioned before, these adjoint functors may be seen as restriction and extension of scalars between module like categories over noncommutative spaces. 

\smallskip
  In Section 3, we describe contramodules over the coalgebra $\mathscr C$ with several objects. Such an object $\matholdcal M$ is given by a collection of vector spaces $\{\matholdcal M(X)\}_{X\in Ob(\mathscr C)}$  equipped with structure maps $Hom_K(\mathscr C(X,Y),\matholdcal M(Y))\longrightarrow \matholdcal M(X)$ for $X$, $Y\in Ob(\mathscr C)$ satisfying conditions given in Definition \ref{D3.1}. We let ${^\mathscr C}Ctr$ denote the category of contramodules over $\mathscr C$. While ${^\mathscr C}Ctr$ need not be a Grothendieck category, we show in Proposition \ref{P3.4} that it has a set of projective generators. An entwined contramodule over $(\mathscr C,A,\psi)$  has structure maps
\begin{equation}
Hom_K(\mathscr C(X,Y),\matholdcal M(Y))\longrightarrow \matholdcal M(X)\qquad \matholdcal M(X)\longrightarrow Hom_K(A,\matholdcal M(X))\qquad X,Y\in Ob(\mathscr C)
\end{equation} as described in Definition \ref{D3.2}. We show in Theorem \ref{T3.8} how to obtain a set of generators for ${_A^{\mathscr C}}Ctr(\psi)$. In Section 3, we also obtain a pair of adjoint functors
\begin{equation}\label{adjsec3}
  S_\psi: {_A^{\mathscr C}}Ctr(\psi)\longrightarrow  {^\mathscr C}Ctr\qquad T_\psi:{^\mathscr C}Ctr\longrightarrow {_A^{\mathscr C}}Ctr(\psi)
\end{equation} As in the case of entwined comodules, we see  these adjoint functors as extension and restriction of scalars  between module like categories over noncommutative spaces. 

\smallskip
We then  give conditions for the functors appearing in \eqref{adjsec1} and \eqref{adjsec3} to be separable.  
  The main results of Sections 4 and 5 can   be summarized into the following two theorems.

\begin{Thm}\label{T1.1introd} (see Theorem \ref{T4.5hm} and Theorem  \ref{T4.5hm5})
Let $\mathscr C$ be a coalgebra with several objects and let $(\mathscr C,A,\psi)$ be an entwining structure. Let $V_1$ be the space such that
an element $\sigma\in V_1$ is a collection of maps $
\sigma=\{\sigma_X:\mathscr C(X,X)\otimes A\longrightarrow K\}_{X\in Ob(\mathscr C)}
$
 such that 
 \begin{equation}\label{4.3zintro}
 \sigma_Y(f_{Y1}\otimes a_{\psi})f_{Y2}^\psi=\sigma_X(f_{X2}\otimes a)f_{X1} \in \mathscr C(X,Y)
 \end{equation} for any $f\in \mathscr C(X,Y)$, $a\in A$, $X$, $Y\in Ob(\mathscr C)$.

\smallskip Then, the following are equivalent:

\smallskip
(1) There exists $\sigma \in V_1$ such that
\begin{equation}\label{421condin}
\sigma_X(g\otimes 1)=\epsilon_X(g)\qquad \forall \textrm{ }g\in \mathscr C(X,X), \textrm{ }X\in Ob(\mathscr C)
\end{equation}

\smallskip
(2) The functor $F_\psi :Com^{\mathscr C}\longrightarrow Com^{\mathscr C}_A(\psi)$ is separable.

\smallskip
(3) The functor $T_\psi :{^{\mathscr C}}Ctr\longrightarrow  {_A^{\mathscr C}}Ctr(\psi)$ is separable.
\end{Thm}

\begin{Thm}\label{T1.2introd} (see Theorem \ref{T4.11hm} and Theorem  \ref{5T4.11hm5})
Let $\mathscr C$ be a coalgebra with several objects and let $(\mathscr C,A,\psi)$ be an entwining structure. Let $W_1$ be the space such that an element
$\lambda\in W_1$ is a collection of linear maps
\begin{equation}
\lambda=\{\mbox{$\lambda^X$ $\vert$ $\lambda^X:\mathscr C(X,X)\longrightarrow A\otimes A$, $\textrm{ }$ $\lambda^{X}(f):= \lambda^{X1}(f)\otimes 
\lambda^{X2}(f)$}\}_{X\in Ob(\mathscr C)}
\end{equation} satisfying the following conditions 
\begin{gather}
\lambda^{Y1}(f_{Y1})\otimes \lambda^{Y2}(f_{Y1})\otimes f_{Y2}=\lambda^{X1}(f_{X2})_\psi\otimes \lambda^{X2}(f_{X2})_\psi\otimes f_{X1}^{\psi\psi} \in A\otimes A\otimes 
\mathscr C(X,Y) \label{425nintro}\\
\lambda^{Z1}(g)\otimes \lambda^{Z2}(g)a=a_\psi\lambda^{Z1}(g^\psi)\otimes \lambda^{Z2}(g^\psi)\in A\otimes A \label{426nintro}
\end{gather} for $f\in \mathscr C(X,Y)$, $g\in \mathscr C(Z,Z)$, $a\in A$, $X,Y,Z\in Ob(\mathscr C)$. 

\smallskip Then, the following are equivalent:

\smallskip
(1) There exists $\lambda \in W_1$ such that
\begin{equation}\label{448condin}
\lambda^{X1}(h)\lambda^{X2}(h)=\epsilon_X(h)\cdot 1 \in A\qquad \forall \textrm{ }h\in \mathscr C(X,X), \textrm{ }X\in Ob(\mathscr C)
\end{equation}

\smallskip
(2) The functor $G_\psi :Com^{\mathscr C}_A(\psi)\longrightarrow Com^{\mathscr C}$   is separable.

\smallskip
(3) The functor $S_\psi : {_A^{\mathscr C}}Ctr(\psi)\longrightarrow  {^{\mathscr C}}Ctr$   is separable. 
\end{Thm} 

We come to Frobenius properties in Section 6. We recall that a pair $(S,T)$ of adjoint functors is said to be Frobenius (see, for instance, \cite{uni}) if $(T,S)$ is also an adjoint pair.  This terminology corresponds to the fact that an extension of rings is Frobenius if and only if restriction and extension of scalars  induced by it form a Frobenius pair of functors. Using again the spaces $V_1$ and $W_1$ from Theorem \ref{T1.1introd} and Theorem \ref{T1.2introd}, we prove the following result for entwined contramodules and entwined comodules over $(\mathscr C,A,\psi)$.

\begin{Thm}\label{T1.3introd} (see Theorem \ref{P6.2} and Theorem \ref{T6.27})  
Let $(\mathscr C,A,\psi)$ be an entwining structure. Then, the following are equivalent:

\smallskip
(1) There exists $\sigma\in V_1$ and $\lambda\in W_1$ such that 
\begin{gather}
\epsilon_X(f)\cdot 1= \sigma_X(f_{X1}\otimes \lambda^{X1}(f_{X2}))\lambda^{X2}(f_{X2})\label{6.4crtintr}\\
\epsilon_X(f)\cdot 1= \sigma_X(f_{X1}^\psi\otimes \lambda^{X2}(f_{X2}))\lambda^{X1}(f_{X2})_\psi\label{6.5crtintr} 
\end{gather} for every $f\in \mathscr C(X,X)$, $X\in Ob(\mathscr C)$. 

\smallskip
(2) The functors \begin{equation}
S_\psi: {_A^{\mathscr C}}Ctr(\psi)\longrightarrow {^{\mathscr C}}Ctr \qquad T_\psi:  {^{\mathscr C}}Ctr
\longrightarrow  {_A^{\mathscr C}}Ctr(\psi)
\end{equation} between categories of   contramodules form a Frobenius pair.

\smallskip
(3) The functors
\begin{equation}
F_\psi: Com^{\mathscr C}\longrightarrow Com_A^{\mathscr C}(\psi) \qquad G_\psi: Com_A^{\mathscr C}(\psi)\longrightarrow Com^{\mathscr C}
\end{equation} between categories of   comodules form a Frobenius pair.
\end{Thm}

In Section 7, we prove Maschke type results for entwined contramodules and entwined comodules over $(\mathscr C,A,\psi)$. In the case of entwined comodules over 
$(\mathscr C,A,\psi)$, we adapt the techniques of Brzezi\'{n}ski \cite{Brz1}. By a normalized cointegral on $(\mathscr C,A,\psi)$, we will mean a collection  of
$K$-linear maps
\begin{equation}\label{norm7intro}
\gamma=\{\gamma_{X}:A^*\otimes\mathscr C(X,X)\longrightarrow A\}_{X\in Ob(\mathscr C)}
\end{equation} satisfying certain conditions that we describe in Definition \ref{D7.1}. We conclude by proving the following result. 

\begin{Thm}\label{T1.4introd} (see Theorem \ref{T7.5} and Theorem \ref{T7.7}) Let $(\mathscr C,A,\psi)$ be an entwining structure. Suppose there exists a normalized cointegral  $\gamma=\{\gamma_{X}:A^*\otimes\mathscr C(X,X)\longrightarrow A\}_{X\in Ob(\mathscr C)}$ on $(\mathscr C,A,\psi)$. Then, we have

\smallskip
(a) The  functor $S_\psi: {_A^{\mathscr C}}Ctr(\psi)\longrightarrow {^{\mathscr C}}Ctr$ on entwined contramodules is a semisimple functor, and also a Maschke functor.

\smallskip
(b) The  functor $G_\psi:  Com_A^{\mathscr C}(\psi)\longrightarrow Com^{\mathscr C}$ on entwined comodules is a semisimple functor, and also a Maschke functor.

\end{Thm}

\section{Entwined comodules over coalgebras with several objects}

Let $K$ be a field and $Vect_K$ be the category of vector spaces over $K$. By a coalgebra with several objects (see \cite{DS}, \cite{Mc1}), we will mean a small category $\mathscr C$ that is enriched over the opposite category of $Vect_K$. More explicitly, we have a set $Ob(\mathscr C)$ of objects as well as $K$-linear maps
\begin{equation}\label{c2.1}
\delta_{XYZ}:\mathscr C(X,Z)\longrightarrow \mathscr C(Y,Z)\otimes \mathscr C(X,Y) \qquad \epsilon_X: \mathscr C(X,X)\longrightarrow K
\end{equation} for every $X$, $Y$, $Z\in Ob(\mathscr C)$ that satisfy ``coassociativity'' and ``counit'' conditions making $\mathscr C$ a category enriched over $Vect_K^{op}$. In other words, we have commutative diagrams
\begin{equation}\label{coalg2cd}
\begin{array}{c}
\begin{CD}
\mathscr C(X,Z) @>\delta_{XYZ}>>  \mathscr C(Y,Z)\otimes \mathscr C(X,Y) \\
@V\delta_{XWZ}VV @VV\delta_{YWZ}\otimes \mathscr C(X,Y)V\\
 \mathscr C(W,Z)\otimes \mathscr C(X,W)@>\mathscr C(W,Z)\otimes \delta_{XYW}>> \mathscr C(W,Z)\otimes \mathscr C(Y,W)\otimes \mathscr C(X,Y)\\
\end{CD}
\\
\\
\xymatrix{
& & \mathscr C(X,Y) \ar[d]^{\delta_{XYY}}\ar[lld]^{id}\\
\mathscr C(X,Y) &&\ar[ll]^{\epsilon_Y\otimes \mathscr C(X,Y)\quad } \mathscr C(Y,Y)\otimes \mathscr C(X,Y)
}\qquad \qquad \qquad 
\xymatrix{
 \mathscr C(X,Y) \ar[d]_{\delta_{XXY}}\ar[rrd]^{id}&& \\
\mathscr C(X,Y)\otimes \mathscr C(X,X)\ar[rr]_{\quad \quad \mathscr C(X,Y)\otimes \epsilon_X}&&\mathscr C(X,Y)\\
}\\
\end{array}
\end{equation} We use Sweedler type notation throughout and suppress the summation to write $\delta_{XYZ}(f):=f_{Y1}\otimes f_{Y2}$ for each $f\in \mathscr C(X,Z)$.  

\smallskip
Let $A$ be a $K$-algebra and suppose we have a collection of $K$-linear maps
\begin{equation}\label{2.2e}
\psi=\{\psi_{XY}:\mathscr C(X,Y)\otimes A\longrightarrow A\otimes \mathscr C(X,Y) \}_{X,Y\in Ob(\mathscr C)}
\end{equation}
For $f\in \mathscr C(X,Y)$ and $a\in A$, we will always write $\psi(f\otimes a)=a_\psi\otimes f^\psi$ by suppressing the summation signs. We will say that the tuple 
$(\mathscr C,A,\psi)$ is an entwining structure if it satisfies the following conditions
\begin{equation}\label{ent2}
\begin{array}{c}
a_\psi\otimes (f^{\psi})_{Y1}\otimes (f^\psi)_{Y2} = a_\psi\otimes \delta_{XYZ}(f^\psi)=\psi_{YZ}(f_{Y1}\otimes a_\psi)\otimes (f_{Y2})^\psi=a_{\psi\psi}\otimes (f_{Y1})^\psi\otimes (f_{Y2})^\psi\\
(ab)_\psi\otimes f^\psi=\psi_{XZ}(f\otimes ab)=a_\psi b_\psi\otimes f^{\psi\psi}\\
\psi_{XZ}(f\otimes 1)=1\otimes f \qquad a_\psi\epsilon_Z(g^\psi)=\epsilon_Z(g)a\\
\end{array}
\end{equation}
 for each $f\in \mathscr C(X,Z)$, $g\in \mathscr C(Z,Z)$, $a$, $b\in A$, $X$, $Y$, $Z\in Ob(\mathscr C)$. 
 
 \smallskip
A (right) $\mathscr C$-comodule $(\mathcal M,\rho^{\mathcal M})$ (see \cite{Mc1}) consists of a collection of vector spaces $\{ \mathcal M(X)\}_{X\in Ob(\mathscr C)}$ along with coaction maps
$\rho_{XY}^{ \mathcal M}: \mathcal M(X)\longrightarrow  \mathcal M(Y)\otimes \mathscr C(X,Y)$ for $X$, $Y\in Ob(\mathscr C)$ such that the following diagrams commute
 \begin{equation}\label{dg2.4}
 \begin{array}{ccc}
 \xymatrix{
 \mathcal M(X) \ar[rr]^{\rho^{ \mathcal M}_{XY}\qquad}\ar[d]_{\rho^{ \mathcal M}_{XZ}} &&  \mathcal M(Y)\otimes \mathscr C(X,Y) \ar[d]^{\mathcal M(Y)\otimes \delta_{XZY}} \\
 \mathcal M(Z)\otimes \mathscr C(X,Z) \ar[rr]^{\rho^{ \mathcal M}_{ZY}\otimes \mathscr C(X,Z)\qquad} &&   \mathcal M(Y)\otimes \mathscr C(Z,Y)\otimes \mathscr C(X,Z) \\
} &&
\xymatrix{
 \mathcal M(X) \ar[d]_{\rho^{ \mathcal M}_{XX}} \ar[drr]^{id}&& \\
 \mathcal M(X)\otimes \mathscr C(X,X) \ar[rr]^{\mathcal M(X)\otimes \epsilon_X}&& \mathcal M(X) \\
}\\
 \end{array}
 \end{equation} While expressing the coaction, we will typically suppress the superscript and the summation and write 
 $\rho^{ \mathcal M}_{XY}(m)=\rho_{XY}(m)=m_{Y0}\otimes m_{Y1} \in  \mathcal M(Y)\otimes  \mathscr C(X,Y)$ for 
 $m\in  \mathcal M(X)$. In this adaptation
 of Sweedler's notation, the commutative diagrams in \eqref{dg2.4} can be expressed as
 \begin{equation}\label{not2.5}
 \begin{array}{c}
 m_{Y0}\otimes m_{Y1Z1}\otimes m_{Y1Z2} =m_{Z0Y0}\otimes m_{Z0Y1}\otimes m_{Z1}\\
  m_{X0} \epsilon_X(m_{X1})=m\\
 \end{array}
 \end{equation} for each $X$, $Y$, $Z\in Ob(\mathscr C)$ and $m\in  \mathcal M(X)$.
 
 \smallskip
 A morphism
 $\phi: \mathcal M\longrightarrow  \mathcal N$ of $\mathscr C$-comodules consists of $k$-linear maps
 $\{\phi(X): \mathcal M(X)\longrightarrow  \mathcal N(X)\}_{X\in Ob(\mathscr C)}$ that is compatible with respective coactions. We will denote by 
 $Com^{\mathscr C}$ the category of right $\mathscr C$-comodules. 
 
 \smallskip
 For $X\in Ob(\mathscr C)$, we let $ \mathcal H_X$ be the right $\mathscr C$-comodule given by setting $ \mathcal H_X(Y):=\mathscr C(Y,X)$ for each 
 $Y\in Ob(\mathscr C)$. The comodule  $ \mathcal H_X$ is referred to as the regular representation of $\mathscr C$ at $X$ (see \cite[Example 2.3]{Mc1}).  For any comodule $(\mathcal M,
 \rho^{\mathcal M})\in Com^{\mathscr C}$, we know from \cite{Mc1} that we have an isomorphism
\begin{equation}\label{Mc111}
Vect_K(\mathcal M(X),K)\cong Com^{\mathscr C}(\mathcal M,\mathcal H_X)
\end{equation}
for each $X\in Ob(\mathscr C)$. More explicitly, given $f\in Vect_K(\mathcal M(X),K)$, we have an induced morphism $\phi_f\in Com^{\mathscr C}(\mathcal M,\mathcal H_X)$ given by $\phi_f(Y):\mathcal M(Y)\xrightarrow{\rho^{\mathcal M}_{YX}}\mathcal M(X)\otimes \mathscr C(Y,X)\xrightarrow{f\otimes 
\mathscr C(Y,X)}\mathcal H_X(Y)$ for each $Y\in Ob(\mathscr C)$. Conversely, given $\phi\in Com^{\mathscr C}(\mathcal M,\mathcal H_X)$, we have the map 
$f: \mathcal M(X)\xrightarrow{\phi(X)}\mathcal H_X(X)=\mathscr C(X,X)\xrightarrow{\epsilon_X}K$ of vector spaces.

 \begin{defn}\label{D2.1}
 Let $(\mathscr C,A,\psi)$ be an entwining structure. An entwined comodule $(\mathcal M,\rho^{\mathcal M},\mu^{\mathcal M})$ over $(\mathscr C,A,\psi)$ consists of the following data:
 
 \smallskip
 (a) A right $\mathscr C$-comodule $( \mathcal M,\rho^{\mathcal M})$
 
 \smallskip
 (b) A collection of right $A$-module structures $\mu^{\mathcal M}=\{\mu_X:\mathcal M(X)\otimes A\longrightarrow \mathcal M(X)\}_{X\in Ob(\mathscr C)}$ such that
 \begin{equation}\label{e2.5}
 (ma)_{Y0}\otimes (ma)_{Y1}=m_{Y0}a_\psi\otimes (m_{Y1})^\psi\in  \mathcal M(Y)\otimes  \mathscr C(X,Y)
 \end{equation} for  $m\in  \mathcal M(X)$, $a\in A$ and $X,Y\in Ob(\mathscr C)$. 
 
 \smallskip
 We will often write an entwined comodule $(\mathcal M,\rho^{\mathcal M},\mu^{\mathcal M})$ simply as $\mathcal M$. A morphism $\phi: \mathcal M\longrightarrow  \mathcal N$ of entwined comodules is a morphism $\phi: \mathcal M\longrightarrow  \mathcal N$ in $Com^{\mathscr C}$ such that 
 $\phi(X): \mathcal M(X)\longrightarrow  \mathcal N(X)$ is $A$-linear for each $X\in Ob(\mathscr C)$. We will denote by $Com^{\mathscr C}_A(\psi)$ the category of entwined comodules. 
 \end{defn}
 
 Since $\otimes$ is exact in the category $Vect_K$, we note that $Com^{\mathscr C}_A(\psi)$ and $Com^{\mathscr C}$ are abelian categories, with kernels, cokernels and direct sums computed objectwise. 
 
 \begin{lem}
 \label{L2.2} Let $ \mathcal M\in Com^{\mathscr C}$. Then,  the collection $ \mathcal M\otimes A:=\{ \mathcal M(X)\otimes A\}_{X\in Ob(\mathscr C)}$ determines an object of $Com^{\mathscr C}_A(\psi)$.
 \end{lem}
 
 \begin{proof}
 Let $m\in  \mathcal M(X)$ and $a\in A$. Then, we set 
 \begin{equation}\label{act1}
 \begin{array}{c}
 (m\otimes a)\cdot b := m\otimes ab \in  \mathcal M(X)\otimes A \qquad \forall\textrm{ }b\in A\\
 (m\otimes a)_{Y0}\otimes (m\otimes a)_{Y1}:= m_{Y0}\otimes a_\psi\otimes ( m_{Y1})^\psi \in  \mathcal M(Y)\otimes A \otimes \mathscr C(X,Y) \qquad\forall\textrm{ }Y\in Ob(\mathscr C)
 \\
 \end{array}
 \end{equation} For $m\in  \mathcal M(X)$ and $a$, $b\in A$, we now verify that
 \begin{equation}
 \begin{array}{ll}
 ( (m\otimes a)\cdot b)_{Y0}\otimes ( (m\otimes a)\cdot b)_{Y1}= (m\otimes ab)_{Y0}\otimes (m\otimes ab)_{Y1}&= m_{Y0}\otimes (ab)_\psi\otimes ( m_{Y1})^\psi \\
 &= m_{Y0}\otimes a_\psi b_\psi\otimes ( m_{Y1})^{\psi\psi} \\
 &=(m_{Y0}\otimes a_\psi)b_\psi\otimes  ( m_{Y1})^{\psi\psi}\\
 &= (m\otimes a)_{Y0}b_\psi\otimes ((m\otimes a)_{Y1})^\psi\\
 \end{array}
 \end{equation} This proves the result. 
 \end{proof}
 
 \begin{lem}
 \label{L2.3} Let $X\in Ob(\mathscr C)$ and let $P\in Mod_A$ be a right $A$-module. Then,  the collection $P\otimes  \mathcal H_X:=\{P\otimes  \mathcal H_X(Y)\}_{Y\in Ob(\mathscr C)}$ determines an object of $Com^{\mathscr C}_A(\psi)$.
 \end{lem}
 \begin{proof}
 Let $p\in P$ and $f\in  \mathcal H_X(Y)=\mathscr C(Y,X)$. Then, we set
 \begin{equation}\label{act2}
 \begin{array}{c}
 (p\otimes f)\cdot b:= pb_\psi \otimes f^\psi  \in P\otimes  \mathcal H_X(Y) \qquad \forall\textrm{ }b\in A
 \\
 (p\otimes f)_{Z0}\otimes (p\otimes f)_{Z1} := p\otimes f_{Z1}\otimes f_{Z2}\in P\otimes  \mathcal H_X(Z)\otimes \mathscr C(Y,Z) \qquad\forall\textrm{ }Z\in Ob(\mathscr C)\\
 \end{array}
 \end{equation} For $p\in P$, $f\in  \mathcal H_X(Y)=\mathscr C(Y,X)$ and $b\in A$, we now verify that
 \begin{equation}
 \begin{array}{ll}
( (p\otimes f)\cdot b)_{Z0}\otimes ( (p\otimes f)\cdot b)_{Z1}= (pb_\psi \otimes f^\psi)_{Z0}\otimes  (pb_\psi \otimes f^\psi)_{Z1}&=pb_\psi\otimes (f^\psi)_{Z1}
\otimes (f^\psi)_{Z2}\\
&= pb_{\psi\psi}\otimes (f_{Z1})^\psi\otimes (f_{Z2})^\psi\\
&= (p\otimes f_{Z1})b_\psi\otimes (f_{Z2})^{\psi}\\
&=(p\otimes f)_{Z0}b_\psi\otimes ((p\otimes f)_{Z1})^\psi\\
 \end{array}
 \end{equation}This proves the result. 
 \end{proof}
 
 For the next result, we pick a $K$-vector space basis for each $\mathscr C(X,Y)$, $X$, $Y\in Ob(\mathscr C)$, take their disjoint union, and denote the collection by $\{h_i\}_{i\in I}$. We now have   the following result (compare \cite[Theorem 2.1.7]{Das}). 
 
 \begin{thm}\label{P2.4}
 Let $ \mathcal M\in Com^{\mathscr C}$ and let $m\in  \mathcal M(X)$ for some $X\in Ob(\mathscr C)$. Then, there exists a subobject $ \mathcal N\subseteq  \mathcal M$ in 
 $Com^{\mathscr C}$ such that $m\in  \mathcal N(X)$ and $ \mathcal N(Y)$ is a finite dimensional $K$-vector space for each $Y\in Ob(\mathscr C)$. 
 \end{thm}
 
 \begin{proof}
 Let $\{\rho_{YZ}: \mathcal M(Y)\longrightarrow  \mathcal M(Z)\otimes \mathscr C(Y,Z)\}_{Y,Z\in Ob(\mathscr C)}$ denote the coaction maps of $ \mathcal M\in Com^{\mathscr C}$.  Using the collection $\{h_i\}_{i\in I}$, we can write for each $Y\in Ob(\mathscr C)$
 \begin{equation}\label{2.11}
 \rho_{XY}(m)=m_{Y0}\otimes m_{Y1}=\underset{i\in I}{\sum} m_i^Y\otimes h_i \in  \mathcal M(Y)\otimes \mathscr C(X,Y)
 \end{equation} In \eqref{2.11}, it is understood that the sum is taken only over those $i\in I$ such that $h_i\in \mathscr C(X,Y)$. 
 
 \smallskip
 For each $Y\in Ob(\mathscr C)$, we now set $ \mathcal N(Y)\subseteq  \mathcal M(Y)$ to be the $K$-subspace generated by the elements $m_i^Y\in  \mathcal M(Y)$ appearing
 in \eqref{2.11}. From \eqref{not2.5}, we notice that $\sum_{i\in I}\epsilon_{X}(h_i)m_i^X=m$, and hence $m\in  \mathcal N(X)$. Using \eqref{dg2.4}, we see that for 
 $Y$, $Z\in Ob(\mathscr C)$, we have
 \begin{equation}\label{212d}
 \underset{i\in I}{\sum} \rho_{YZ}(m_i^Y)\otimes h_i = \underset{i\in I}{\sum} m_i^Z\otimes \delta_{XYZ}( h_i)=\underset{i,j,k\in I}{\sum}t_i(j,k) m_i^Z\otimes h_j\otimes h_k \in  \mathcal M(Z)\otimes \mathscr C(Y,Z)\otimes \mathscr C(X,Y)
 \end{equation} where the coefficients $t_i(j,k)\in K$. From \eqref{212d}, it follows that 
 $\rho_{YZ}(m^Y_i)=\underset{i,j\in I}{\sum} t_i(j,i)m_i^Z\otimes h_j\in  \mathcal M(Z)\otimes \mathscr C(Y,Z)$. It follows that the coactions $\{\rho_{YZ}: \mathcal M(Y)\longrightarrow  \mathcal M(Z)\otimes \mathscr C(Y,Z)\}_{Y,Z\in Ob(\mathscr C)}$ restrict to $\{\rho_{YZ}: \mathcal N(Y)\longrightarrow  \mathcal N(Z)\otimes \mathscr C(Y,Z)\}_{Y,Z\in Ob(\mathscr C)}$. This proves the result.
 \end{proof}
 
From now onwards, we will say that an object $ \mathcal N\in Com^{\mathscr C}$ is finitary if $ \mathcal N(X)$ is finite dimensional as a $K$-vector space for each 
$X\in Ob(\mathscr C)$. 

\begin{Thm}\label{T2.5}
Let $(\mathscr C,A,\psi)$ be an entwining structure. Then, the category $Com_A^{\mathscr C}(\psi)$ of entwined comodules is a Grothendieck category. 
\end{Thm}

\begin{proof}
We consider $\mathcal M\in Com^{\mathscr C}_A(\psi)$ as well as $m\in  \mathcal M(X)$ for some $X\in Ob(\mathscr C)$. By Proposition \ref{P2.4}, there exists a subobject $ \mathcal N\subseteq  \mathcal M$ in 
 $Com^{\mathscr C}$ such that $ \mathcal N$ is finitary and $m\in  \mathcal N(X)$. Using Lemma \ref{L2.2}, we have the object $ \mathcal N\otimes A\in  Com^{\mathscr C}_A(\psi)$. We now define
 \begin{equation}
 \phi(Y): \mathcal N(Y)\otimes A\longrightarrow  \mathcal M(Y) \qquad n\otimes a\mapsto na
 \end{equation} for each $Y\in Ob(\mathscr C)$. It is immediate that each $\phi(Y)$ is right $A$-linear. Since $m\in  \mathcal N(X)$, it is also clear that $m\in Im(\phi(X))$. For $Y$, $Z
 \in Ob(\mathscr C)$, we now see that
 \begin{equation}\label{214}
 \begin{array}{ll}
 (\phi(Y)(n\otimes a))_{Z0}\otimes  (\phi(Y)(n\otimes a))_{Z1}=(na)_{Z0}\otimes (na)_{Z1} &=n_{Z0}a_\psi\otimes (n_{Z1})^\psi\\
 &=\phi(Z)(n_{Z0}\otimes a_\psi )\otimes (n_{Z1})^\psi \\
 &=\phi(Z)((n\otimes a)_{Z0})\otimes (n\otimes a)_{Z1} \\
 \end{array}
 \end{equation} It follows from \eqref{214} that $\phi=\{\phi(Y)\}_{Y\in Ob(\mathscr C)}$ is a morphism in $Com_A^{\mathscr C}(\psi)$.  We now see that the collection of isomorphism classes 
 $\{ \mathcal N'\otimes A\}$ where $ \mathcal N'$ varies over all finitary objects in $Com^{\mathscr C}$ gives a set of generators for $Com_A^{\mathscr C}(\psi)$.  Since filtered colimits and finite limits in 
 $Com_A^{\mathscr C}(\psi)$ are both computed objectwise, we know that $Com_A^{\mathscr C}(\psi)$ satisfies the (AB5) condition. Hence,   $Com_A^{\mathscr C}(\psi)$  is a Grothendieck category.
\end{proof}

From the proof of Lemma \ref{L2.2}, it is clear that we have a functor
\begin{equation}\label{215t}
F_\psi :Com^{\mathscr C}\longrightarrow Com_A^{\mathscr C}(\psi)\qquad  \mathcal M\mapsto  \mathcal M\otimes A
\end{equation} We also  denote by $G_\psi $ the forgetful functor from $ Com_A^{\mathscr C}(\psi)$ to $Com^{\mathscr C}$. 

\begin{thm}\label{P2.6}
Let $(\mathscr C,A,\psi)$ be an entwining structure. Then, $(F_\psi ,G_\psi )$ is a pair of adjoint functors, i.e., we have natural isomorphisms
\begin{equation}\label{216g}
 Com_A^{\mathscr C}(\psi)( \mathcal M\otimes A, \mathcal N)\cong Com^{\mathscr C}( \mathcal M, \mathcal N)
\end{equation} for $ \mathcal M\in Com^{\mathscr C}$, $ \mathcal N\in Com^{\mathscr C}_A(\psi)$. 
\end{thm}
 \begin{proof}
We consider $\eta\in  Com_A^{\mathscr C}(\psi)( \mathcal M\otimes A, \mathcal N)$. For each $X\in Ob(\mathscr C)$, we set
\begin{equation}
\zeta(X): \mathcal M(X)\longrightarrow \mathcal N(X)\qquad m\mapsto \eta(X)(m\otimes 1)
\end{equation} For $X$, $Y\in Ob(\mathscr C)$ and $m\in  \mathcal M(X)$, we verify that
\begin{equation}\label{218bg}
\begin{array}{ll}
(\zeta(X)(m))_{Y0}\otimes (\zeta(X)(m))_{Y1}= ( \eta(X)(m\otimes 1))_{Y0}\otimes ( \eta(X)(m\otimes 1))_{Y1}&=\eta(Y)((m\otimes 1)_{Y0})\otimes (m\otimes 1)_{Y1}\\
&=\eta(Y)(m_{Y0}\otimes 1)\otimes m_{Y1}\\
&=\zeta(Y)(m_{Y0})\otimes m_{Y1}\\
\end{array}
\end{equation} It follows from \eqref{218bg} that $\zeta=\{\zeta(X)\}_{X\in Ob(\mathscr C)}$ gives a morphism in $Com^{\mathscr C}$.

\smallskip
Conversely, we consider $\zeta\in Com^{\mathscr C}( \mathcal M, \mathcal N)$. For each $X\in Ob(\mathscr C)$, we set
\begin{equation}\label{219c}
\eta(X): \mathcal M(X)\otimes A\longrightarrow  \mathcal N(X)\qquad (m\otimes a)\mapsto \zeta(X)(m)a
\end{equation} It is immediate that $\eta(X)$ is right $A$-linear.  For $X$, $Y\in Ob(\mathscr C)$ and $m\in  \mathcal M(X)$, $a\in A$, we verify that
\begin{equation}\label{220bg}
\begin{array}{ll}
(\eta(X)(m\otimes a))_{Y0}\otimes (\eta(X)(m\otimes a))_{Y1}&=(\zeta(X)(m)a)_{Y0}\otimes (\zeta(X)(m)a)_{Y1}\\
&=(\zeta(X)(m))_{Y0}a_\psi\otimes ((\zeta(X)(m))_{Y1})^\psi \\
&=\zeta(X)(m_{Y0})a_\psi\otimes (m_{Y1})^\psi\\
&=\eta(X)(m_{Y0}\otimes a_\psi)\otimes (m_{Y1})^\psi\\
&=\eta(X)((m\otimes a)_{Y0})\otimes (m\otimes a)_{Y1}\\
\end{array}
\end{equation} It follows from \eqref{220bg} that $\eta=\{\eta(X)\}_{X\in Ob(\mathscr C)}$ gives a morphism in $Com^{\mathscr C}_A(\psi)$.
 \end{proof}

 \section{Entwined contramodules over coalgebras with several objects}
 
 We continue with $\mathscr C$ being a coalgebra with several objects.  The notion of contramodule over a coalgebra with several objects is implicit in the literature (see also \cite{Posi}), but we describe it here. If $U'$, $U$ are vector spaces over $K$, we will denote by $(U',U)$ the space of $K$-linear maps from $U'$ to $U$. 
 
 \begin{defn}\label{D3.1}
 Let $\mathscr C$ be a coalgebra with several objects. A left contramodule $
( \matholdcal M,\pi^{\matholdcal M})$ over $\mathscr C$ consists of the following data:
 
 \smallskip
 (a) A vector space $\matholdcal M(X)$ for each $X\in Ob(\mathscr C)$. 
 
 \smallskip
 (b) Maps $\pi_{XY}^{\matholdcal M}=\pi_{XY}: (\mathscr C(X,Y),\matholdcal M(Y))\longrightarrow \matholdcal M(X)$ for each pair $X$, $Y$ of objects in $\mathscr C$  such that the following diagrams commute
 \begin{equation}\label{dg3.4}
 \begin{array}{c}
 \xymatrix{
(\mathscr C(X,Y), (\mathscr C(Y,Z),\matholdcal M(Z)))\cong (\mathscr C(Y,Z)\otimes \mathscr C(X,Y),\matholdcal M(Z)) \ar[rrr]^{\qquad\qquad\qquad \qquad(\delta_{XYZ},\matholdcal M(X))}\ar[d]_{(\mathscr C(X,Y),\pi_{YZ})} &&& (\mathscr C(X,Z),\matholdcal M(Z)) \ar[d]^{\pi_{XZ}} \\
(\mathscr C(X,Y),\matholdcal M(Y))\ar[rrr]^{\pi_{XY}} &&&  \matholdcal M(X) \\
} \\ \\
\xymatrix{
\matholdcal M(X) \cong (K,\matholdcal M(X))\ar[d]_{(\epsilon_X,\matholdcal M(X))} \ar[drr]^{id}&& \\
(\mathscr C(X,X),\matholdcal M(X)) \ar[rr]^{\pi_{XX}}&&\matholdcal M(X) \\
}\\
 \end{array}
 \end{equation}
 A morphism $\phi:( \matholdcal M,\pi^{\matholdcal M})\longrightarrow ( \matholdcal N,\pi^{\matholdcal N})$ of $\mathscr C$-contramodules consists of linear maps
 $\{\phi(X):\matholdcal M(X)\longrightarrow \matholdcal N(X)\}_{X\in Ob(\mathscr C)}$ such that $\phi(X)\circ \pi_{XY}^{\matholdcal M}=\pi_{XY}^{\matholdcal N}\circ (\mathscr C(X,Y),\phi(Y))$ for all $X$, $Y\in Ob(\mathscr C)$. We will denote by ${^\mathscr C}Ctr$ the category of left $\mathscr C$-contramodules. 
 \end{defn}

For an ordinary coalgebra entwined with an algebra, the notion of entwined contramodule appears in \cite[$\S$ 10]{Pos}. We describe below the extension of this notion to coalgebras with several objects. While working with contramodules, we will often write the structure map on a left $A$-module $M$ as a linear map $M\longrightarrow (A,M)$ instead of the more common $A\otimes M \longrightarrow M$. 
 
 \begin{defn}\label{D3.2}
 Let $\mathscr C$ be a coalgebra with several objects and let $(\mathscr C,A,\psi)$ be an entwining structure.  An entwined contramodule 
 $(\matholdcal M,\pi^{\matholdcal M},\mu^{\matholdcal M})$ over $(\mathscr C,A,\psi)$ consists of the following data:
 
 \smallskip
 (a) A left $\mathscr C$-contramodule $(\matholdcal M,\pi^{\matholdcal M})$
 
 \smallskip
 (b) A collection of  left $A$-module structures $\mu^{\matholdcal M}=\{\mu_X: \matholdcal M(X)\longrightarrow (A,\matholdcal M(X))\}_{X\in Ob(\mathscr C)}$   such that the following diagram commutes
 \begin{equation}\label{com3.2}
 \xymatrix{
  (\mathscr C(X,Y),\matholdcal M(Y))\ar[r]^{\pi_{XY}}\ar[d]_{(\mathscr C(X,Y),\mu_Y)}& \matholdcal M(X) \ar[r]^{\mu_X} & (A,\matholdcal M(X))\\
   (\mathscr C(X,Y),(A,\matholdcal M(Y))) \cong (A\otimes \mathscr C(X,Y),\matholdcal M(Y))\ar[rr]_{(\psi_{XY},\matholdcal M(Y))}&&( \mathscr C(X,Y)\otimes A,\matholdcal M(Y))\cong  (A,(\mathscr C(X,Y),\matholdcal M(Y)))\ar[u]_{(A,\pi_{XY})}\\
 }
 \end{equation}  
 A morphism $\phi: \matholdcal M\longrightarrow  \matholdcal N$ of entwined contramodules is a morphism $\phi: \matholdcal M\longrightarrow  \matholdcal N$ in ${^\mathscr C}Ctr$ such that 
 $\phi(X): \matholdcal M(X)\longrightarrow  \matholdcal N(X)$ is $A$-linear for each $X\in Ob(\mathscr C)$. We will denote by ${_A^{\mathscr C}}Ctr(\psi)$ the category of entwined contramodules. 
 \end{defn}
 
 Since  Hom and $\otimes$ are exact in $Vect_K$, we observe that ${^\mathscr C}Ctr$ and ${_A^{\mathscr C}}Ctr(\psi)$ are abelian categories, with finite limits and finite colimits computed objectwise in the category of vector spaces. 
 
 \smallskip
 
Let $U$ be a $K$-vector space. For  $X\in Ob(\mathscr C)$, we denote by $\matholdcal H_X^U$ the $\mathscr C$-contramodule given by setting $\matholdcal H_X^U(Y):=(\mathscr C(Y,X),U)$. The structure maps for $\matholdcal H_X^U$ are given  by
\begin{equation}\label{33tr}
(\mathscr C(Y,Z),\matholdcal H_X^U(Z))\cong (\mathscr C(Y,Z),(\mathscr C(Z,X),U))\cong (\mathscr C(Z,X)\otimes \mathscr C(Y,Z),U)\xrightarrow{\qquad (\delta_{YZX},U)\qquad }
(\mathscr C(Y,X),U)=\matholdcal H_X^U(Y)
\end{equation} for $Y$, $Z\in Ob(\mathscr C)$.  When $U=K$, we write $\matholdcal H_X^*:=\matholdcal H_X^K$. 

\begin{thm}\label{P3.3}
Let $X\in Ob(\mathscr C)$. Consider the functors 
\begin{equation}\label{eq3.4f}
\begin{array}{c} \matholdcal H_X:Vect_K\longrightarrow {^\mathscr C}Ctr\qquad U\mapsto \matholdcal H_X^U \\
ev_X:{^\mathscr C}Ctr\longrightarrow Vect_K \qquad \matholdcal M\mapsto \matholdcal M(X)\\
\end{array}
\end{equation} Then, $(\matholdcal H_X,ev_X)$ is a pair of adjoint functors. 
\end{thm}

\begin{proof}
We consider $U\in Vect_K$ and $\matholdcal M\in {^\mathscr C}Ctr$. We have to show that there are natural isomorphisms
\begin{equation}\label{3.5rf}
{^\mathscr C}Ctr(\matholdcal H_X^U,\matholdcal M)\cong Vect_K(U,\matholdcal M(X))
\end{equation} Let $\eta:
U\longrightarrow \matholdcal M(X)$ be a morphism in $Vect_K$. Then, for each $Y\in Ob(\mathscr C)$, we have an induced morphism
\begin{equation}\label{3.6t}
\zeta(Y):\matholdcal H_X^U(Y)=(\mathscr C(Y,X),U)\xrightarrow{\quad(\mathscr C(Y,X),\eta)\quad}(\mathscr C(Y,X),\matholdcal M(X))\xrightarrow{\quad \pi_{YX}\quad}\matholdcal M(Y)
\end{equation} where $\pi_{YX}$ comes from the structure maps of $\matholdcal M$ as a $\mathscr C$-contramodule. Together, the morphisms in \eqref{3.6t} determine 
$\zeta\in {^\mathscr C}Ctr(\matholdcal H_X^U,\matholdcal M)$.

\smallskip
Conversely, we consider some $\zeta\in {^\mathscr C}Ctr(\matholdcal H_X^U,\matholdcal M)$. Then, we have an induced map
\begin{equation}
\eta: U=(K,U)\xrightarrow{\quad (\epsilon_X,U)\quad} (\mathscr C(X,X),U)=\matholdcal H_X^U(X)\xrightarrow{\quad\zeta(X)\quad}\matholdcal M(X)
\end{equation} in $Vect_K$. It may be verified that these associations are inverse to each other. This proves the result. 
\end{proof}

\begin{thm}\label{P3.4}
The collection $\{\matholdcal H_X^*\}_{X\in Ob(\mathscr C)}$ is a set of projective generators for ${^\mathscr C}Ctr$. 
\end{thm}

\begin{proof}
We have mentioned before that kernels and cokernels are computed objectwise in ${^\mathscr C}Ctr$. From Proposition \ref{P3.3}, it follows that 
${^\mathscr C}Ctr(\matholdcal H_X^*,\matholdcal M)\cong Vect_K(K,\matholdcal M(X))\cong \matholdcal M(X)$ for any $\matholdcal M\in {^\mathscr C}Ctr$. From this, it is clear that each 
$\matholdcal H_X^*$ is projective. 

\smallskip
It remains to show that $\{\matholdcal H_X^*\}_{X\in Ob(\mathscr C)}$ is a set of generators. For this, we consider a monomorphism $\iota: \matholdcal N\hookrightarrow \matholdcal M$ in ${^\mathscr C}Ctr$ that is not an isomorphism. Since kernels in ${^\mathscr C}Ctr$ are computed objectwise, this means there exists some $X\in Ob(\mathscr C)$ such that
$\iota(X):\matholdcal N(X)\hookrightarrow \matholdcal M(X)$ is not an isomorphism. In other words, we can choose a morphism $\eta:K\longrightarrow \matholdcal M(X)$ in $Vect_K$ that does not factor through $\matholdcal N(X)$. Then, the morphism $\zeta\in {^\mathscr C}Ctr(\matholdcal H_X^*,\matholdcal M)\cong Vect_K(K,\matholdcal M(X))$ corresponding to 
$\eta$ does not factor through $\matholdcal N$. It now follows by \cite[$\S$ 1.9]{Toh} that $\{\matholdcal H_X^*\}_{X\in Ob(\mathscr C)}$ is a set of  generators for ${^\mathscr C}Ctr$. 
\end{proof}

\begin{cor}\label{C3.5}
Let $\matholdcal M\in {^\mathscr C}Ctr$. Let $m\in \matholdcal M(X)$ for some $X\in Ob(\mathscr C)$ and let $\phi_m: \matholdcal H_X^*\longrightarrow \matholdcal M$ be the morphism 
in ${^\mathscr C}Ctr$ corresponding to $m$. Then, $m$ lies in the image of the map $\phi_m(X):\matholdcal H_X^*(X)\longrightarrow \matholdcal M(X)$.
\end{cor}

\begin{proof}
From the construction in \eqref{3.6t}, we see  that $\phi_m(X)(g):=\pi_{XX}(g^m)$, where $g\in (\mathscr C(Y,X),K)$ and $g^m:\mathscr C(X,X)\longrightarrow \matholdcal M(X)$ is given by 
setting $g^m(f):=g(f)m$. In particular, we take $g:=\epsilon_X :\mathscr C(X,X)\longrightarrow K$. Then, $g^m(f)=\epsilon_X^m(f):=\epsilon_X(f)m$. From the counit condition in \eqref{dg3.4}, we know that $\pi_{XX}(\epsilon_X^m)=m$. It is now clear that $\phi_m(X)(\epsilon_X)=\pi_{XX}(\epsilon_X^m)=m$. 
\end{proof}
 
 We now define some notation. If $L_1$,...,$L_k$, $k\geq 1$ is a sequence of vector spaces, we write
 \begin{equation}\label{ntn3}
 (L_k,L_{k-1},...,L_1):=(L_k,(L_{k-1},(...(L_2,L_1))))
 \end{equation} Further, we will make no distinction between all the different ways in which the right hand side of \eqref{ntn3} can be written using the hom-tensor adjunction. For example, 
 $(L_k,L_{k-1},...,L_1)$ will also be used to denote $(L_k,(L_{k-2}\otimes L_{k-1},(L_{k-3},(...(L_2,L_1)))))$, or $(L_k,L_{k-1},(...(L_3\otimes L_4\otimes L_5,(L_2,L_1))))$ and so on. 
 
 \smallskip
Given an entwining structure $(\mathscr C,A,\psi)$, we see that  the maps $\psi=\{\psi_{YZ}:\mathscr C(Y,Z)\otimes A\longrightarrow A\otimes \mathscr C(Y,Z) \}_{Y,Z\in Ob(\mathscr C)}$ induce morphisms
\begin{equation}\label{39uh}
\begin{CD}
(\mathscr C(Y,Z),A,U)=(A\otimes \mathscr C(Y,Z),U)@>(\psi_{YZ},U)>> (\mathscr C(Y,Z)\otimes A,U)=(A,\mathscr C(Y,Z),U)
\end{CD}
\end{equation}
for each vector space $U$. From the properties in \eqref{ent2}, it follows that the maps in \eqref{39uh} satisfy the following conditions
\begin{equation}\label{homent3}
\xymatrix{(\mathscr C(Y,W), \mathscr C(W,Z),A,U)\ar[rr]^{\qquad (\delta_{YWZ},A,U) }\ar[d]_{(\mathscr C(Y,W),\psi_{WZ},U)}&&(\mathscr C(Y,Z),A,U)\ar[rr]^{(\psi_{YZ},U)} && (A,\mathscr C(Y,Z),U) \\
( \mathscr C(Y,W),A,\mathscr C(W,Z), U)\ar[rrrr]_{(\psi_{YW},\mathscr C(W,Z),U)}&&&&(A, \mathscr C(Y,W),\mathscr C(W,Z), U)\ar[u]_{(A,\delta_{YWZ},U)}\\
}
\end{equation} for any $Y$, $Z$, $W\in Ob(\mathscr C)$. Similarly, if $\mu_A:A\otimes A\longrightarrow A$ denotes the multiplication map, it follows from the properties in 
\eqref{ent2} that the following diagram commutes 
\begin{equation}\label{homent4}
\xymatrix{(\mathscr C(Y,Z),A,U)\ar[rr]^{\qquad (\psi_{YZ},U) }\ar[d]_{(\mathscr C(Y,Z),\mu_A,V)\qquad }&&(A,\mathscr C(Y,Z),U)\ar[rr]^{(\mu_A,\mathscr C(Y,Z),U)} && (A,A,\mathscr C(Y,Z),U) \\
( \mathscr C(Y,Z),A,A, U)\ar[rrrr]_{(\psi_{YZ},A,U)}&&&&(A, \mathscr C(Y,Z),A, V)\ar[u]_{(A,\psi_{YZ},U)}\\
}
\end{equation}
for any $Y$, $Z\in Ob(\mathscr C)$.
 
 \begin{lem}\label{L3.5}
 Let $(\mathscr C,A,\psi)$ be an entwining structure and let $P\in {_A}Mod$ be left $A$-module. We fix $X\in Ob(\mathscr C)$. Then, setting $\matholdcal H_X^P(Y):=(\mathscr C(Y,X),P)$ for
 $Y\in Ob(\mathscr C)$  determines an entwined contramodule. Further, this gives a functor from ${_A}Mod$ to ${_A^{\mathscr C}}Ctr(\psi)$.  
 \end{lem}
 \begin{proof}
 Let $\mu_P:P\longrightarrow (A,P)$ denote the structure map of $P$ as a left $A$-module. The $\mathscr C$-contramodule structure on $\matholdcal H_X^P$ follows as in 
 \eqref{33tr}. The $A$-module structure on $\matholdcal H_X^P(Y)$ is determined by
 \begin{equation}\label{39k}
 \begin{CD}
 \matholdcal H_X^P(Y)=(\mathscr C(Y,X),P)@>(\mathscr C(Y,X),\mu_P)>> (\mathscr C(Y,X),A,P)@>(\psi_{YX},P)>>(A,\mathscr C(Y,X),P)=(A, \matholdcal H_X^P(Y))
 \end{CD}
 \end{equation} It remains to show that $\matholdcal H_X^P$ satisfies the conditions in \eqref{dg3.4}. It follows from \eqref{homent3} that the following maps 
 from $(\mathscr C(Y,Z),\matholdcal H_X^P(Z))$ to $(A,\matholdcal H_X^P(Y))$ are equal:
 \begin{equation}\label{40j}
 \begin{array}{l}
(A,\delta_{YZX},P)\circ (\psi_{YZ},\mathscr C(Z,X),P)\circ (\mathscr C(Y,Z),\psi_{ZX},P)\circ  (\mathscr C(Y,Z),(\mathscr C(Z,X),\mu_P))\\
=(\psi_{YX},P)\circ (\delta_{YZX},A,P)\circ  (\mathscr C(Y,Z),(\mathscr C(Z,X),\mu_P))\\
=(\psi_{YX},P)\circ (\delta_{YZX},(A,P))\circ  (\mathscr C(Z,X)\otimes \mathscr C(Y,Z),\mu_P)\\
=(\psi_{YX},P)\circ (\mathscr C(Y,X),\mu_P)\circ  (\delta_{YZX},P)\\
 \end{array}
 \end{equation} This proves the result.
 \end{proof}

 \begin{lem}\label{L3.6}
 Let $(\mathscr C,A,\psi)$ be an entwining structure and let $\matholdcal M\in {^\mathscr C}Ctr$.   Then, we have a functor
 \begin{equation}
 T_\psi:{^\mathscr C}Ctr\longrightarrow {_A^{\mathscr C}}Ctr(\psi) \qquad T_\psi(\matholdcal M)(X):=(A,\matholdcal M(X)) \quad X\in Ob(\mathscr C)
 \end{equation}  
 \end{lem}
 \begin{proof} If $\mu_A:A\otimes A\longrightarrow A$ denotes the multiplication on $A$, it is clear that $(A,\matholdcal M(X))\xrightarrow{(\mu_A,\matholdcal M(X))}(A,(A,
 \matholdcal M(X)))$ gives each $ (A,\matholdcal M(X)) $ the structure of a left $A$-module. Let $\pi_{XY}: (\mathscr C(X,Y),\matholdcal M(Y))\longrightarrow \matholdcal M(X)$  denote the structure maps of $\matholdcal M$ as a $\mathscr C$-contramodule.  The $\mathscr C$-contramodule structure on $T_\psi(\matholdcal M)$ now follows from the structure maps
 \begin{equation}\label{45k}
 (\mathscr C(X,Y),T_\psi(\matholdcal M)(Y))=(\mathscr C(X,Y),A,\matholdcal M(Y))\xrightarrow{(\psi_{XY},\matholdcal M(Y))} (A,\mathscr C(X,Y),\matholdcal M(Y))
 \xrightarrow{(A,\pi_{XY})}(A,\matholdcal M(X))
 \end{equation}
 for $X$, $Y\in Ob(\mathscr C)$. It remains to show that $T_\psi(\matholdcal M)$ satisfies the conditions in \eqref{dg3.4}. It follows from \eqref{homent4} that the following maps 
 from $(\mathscr C(X,Y),T_\psi(\matholdcal M)(Y))$ to $(A,T_\psi(\matholdcal M)(X))$ are equal:
 \begin{equation}
 \begin{array}{l}
(A,A,\pi_{XY})\circ (A,\psi_{XY},\matholdcal M(Y))\circ (\psi_{XY},A,\matholdcal M(Y))\circ (\mathscr C(X,Y),\mu_A,\matholdcal M(Y))\\
=(A,A,\pi_{XY})\circ (\mu_A,\mathscr C(X,Y),\matholdcal M(Y))\circ (\psi_{XY},\matholdcal M(Y))\\
=(\mu_A,\matholdcal M(X))\circ (A,\pi_{XY})\circ (\psi_{XY},\matholdcal M(Y))\\
 \end{array}
 \end{equation}This proves the result.
 \end{proof}
 
 \begin{thm}\label{P3.7}
 Let $(\mathscr C,A,\psi)$ be an entwining structure and let $ S_\psi: {_A^{\mathscr C}}Ctr(\psi)\longrightarrow {^\mathscr C}Ctr$ be the forgetful functor. Then, $( S_\psi, T_\psi)$ is a pair of adjoint functors.
 \end{thm}
 
 \begin{proof}
 We consider $\matholdcal M\in {_A^{\mathscr C}}Ctr(\psi)$, $\matholdcal M'\in {^\mathscr C}Ctr$  and $\eta\in {^\mathscr C}Ctr(S_\psi(\matholdcal M),\matholdcal M')$. We now define $\zeta:\matholdcal M\longrightarrow T_\psi(
 \matholdcal M')$ by setting for each $X\in Ob(\mathscr C)$:
 \begin{equation}\label{317g}
\zeta(X):\matholdcal M(X)\xrightarrow{\mu_X} (A,\matholdcal M(X)) \xrightarrow{(A,\eta(X))} (A,\matholdcal M'(X))
 \end{equation} where $\mu_X:\matholdcal M(X)\longrightarrow (A,\matholdcal M(X))$ is the structure map of $\matholdcal M(X)$ as an $A$-module. The map in \eqref{317g} is clearly 
 $A$-linear.  Let $\pi_{XY}: (\mathscr C(X,Y),\matholdcal M(Y))\longrightarrow \matholdcal M(X)$ (resp.  $\pi'_{XY}: (\mathscr C(X,Y),\matholdcal M'(Y))\longrightarrow \matholdcal M'(X)$ )  denote the structure maps of $\matholdcal M$ (resp. $\matholdcal M'$) as a $\mathscr C$-contramodule. To show that the maps in \eqref{317g} are compatible with the  $\mathscr C$-contramodule structures on $\matholdcal M$ and $T_\psi(\matholdcal M')$, we check that for $X$, $Y\in Ob(\mathscr  C)$, we have
 \begin{equation}
 \begin{array}{l}
  (A,\pi'_{XY})\circ (\psi_{XY},\matholdcal M'(Y))\circ (\mathscr C(X,Y),\zeta(Y))\\
 (A,\pi'_{XY})\circ (\psi_{XY},\matholdcal M'(Y))\circ (\mathscr C(X,Y),A,\eta(Y))\circ (\mathscr C(X,Y),\mu_Y)\\
 = (A,\pi'_{XY})\circ (A,\mathscr C(X,Y),\eta(Y)) \circ (\psi_{XY},\matholdcal M(Y))\circ   (\mathscr C(X,Y),\mu_Y)\\
  =  (A,\eta(X))\circ (A,\pi_{XY}) \circ (\psi_{XY},\matholdcal M(Y))\circ  (\mathscr C(X,Y),\mu_Y)\\
=    (A,\eta(X))\circ  \mu_X\circ \pi_{XY}=\zeta(X)\circ \pi_{XY}\\
 \end{array}
 \end{equation} Conversely, we take $\zeta\in {_A^{\mathscr C}}Ctr(\psi)(\matholdcal M, T_\psi(
 \matholdcal M')$. We define $\eta: S_\psi(\matholdcal M)\longrightarrow \matholdcal M'$ by setting for each $X\in Ob(\mathscr C)$:
 \begin{equation}\label{318g}
 \eta(X):\matholdcal M(X)\xrightarrow{\zeta(X)}(A,\matholdcal M'(X))\xrightarrow{(u_A,\matholdcal M'(X))}\matholdcal M'(X)
 \end{equation} where $u_A:K\longrightarrow A$ is the unit map of the algebra $A$. To show that the maps in \eqref{318g} are compatible with the  $\mathscr C$-contramodule structures on $\matholdcal M$ and $\matholdcal M'$, we check that for $X$, $Y\in Ob(\mathscr  C)$, we have
 \begin{equation}
 \begin{array}{ll}
\eta(X)\circ \pi_{XY}&= (u_A,\matholdcal M'(X))\circ \zeta(X)\circ \pi_{XY}\\
& = (u_A,\matholdcal M'(X))\circ  (A,\pi'_{XY})\circ (\psi_{XY},\matholdcal M'(Y))\circ (\mathscr C(X,Y),\zeta(Y))\\
 & =\pi'_{XY}\circ (u_A,\mathscr C(X,Y),\matholdcal M'(Y))\circ (\psi_{XY},\matholdcal M'(Y))\circ (\mathscr C(X,Y),\zeta(Y))\\
  & =\pi'_{XY}\circ (\mathscr C(X,Y),u_A,\matholdcal M'(Y))\circ (\mathscr C(X,Y),\zeta(Y))=\pi'_{XY}\circ (\mathscr C(X,Y),\eta(Y))\\
 \end{array}
 \end{equation} It may be verified that these two associations are inverse to each other. This proves the result.
 \end{proof}
 
 \begin{lem}\label{L3.75}
  Let $(\mathscr C,A,\psi)$ be an entwining structure. Then,  the category ${_A^{\mathscr C}}Ctr(\psi)$ of entwined contramodules is well powered, i.e., the collection of subobjects
  of any object forms a set. 
 \end{lem}
 
 \begin{proof}
 Let $\phi: (\matholdcal M',\pi',\mu')\hookrightarrow (\matholdcal M,\pi,\mu)$ be a monomorphism in ${_A^{\mathscr C}}Ctr(\psi)$. In particular,  $\phi(X)$ is a monomorphism for 
 each $X\in Ob(\mathscr C)$. Then for $X$, $Y\in Ob(\mathscr C)$, we see that $\pi'_{XY}: (\mathscr C(X,Y),\matholdcal M'(Y))\longrightarrow \matholdcal M'(X)$ is the unique morphism satisfying   $\phi(X)\circ \pi'_{XY}=\pi_{XY}\circ (\mathscr C(X,Y),\phi(Y))$. Also, $(A,\phi(X)):(A,\matholdcal M'(X))\longrightarrow (A,\matholdcal M(X))$ must be a monomorphism for 
 each $X\in Ob(\mathscr C)$. Hence, $\mu'_X:\matholdcal M'(X)\longrightarrow (A,\matholdcal M'(X))$ is the unique morphism satisfying $(A,\phi(X))\circ \mu'_X=\mu_X\circ \phi(X)$. Accordingly, the subobject $\matholdcal M'\subseteq 
 \matholdcal M$ is completely determined by the collection of $K$-subspaces $\{\matholdcal M'(X)\subseteq \matholdcal M(X)\}_{X\in Ob(\mathscr C)}$. This proves the result. 
 \end{proof}
 
 \begin{Thm}\label{T3.8}
  Let $(\mathscr C,A,\psi)$ be an entwining structure. Then, the category ${_A^{\mathscr C}}Ctr(\psi)$ of entwined contramodules  over $(\mathscr C,A,\psi)$ has a set of generators. 
 \end{Thm}
 
 \begin{proof}
 We consider $\matholdcal M\in {_A^{\mathscr C}}Ctr(\psi)$ and choose $m\in \matholdcal M(X)$ for some $X\in Ob(\mathscr C)$. From Corollary \ref{C3.5}, we know that 
 $m\in \matholdcal M(X)$ corresponds to a morphism
$\phi_m:\matholdcal H_X^*\longrightarrow \matholdcal M$ in ${^\mathscr C}Ctr$ such that $m$ lies in the image of $\phi_m(X)$.  By Lemma 
\ref{L3.6}, we also have an induced morphism 
$(A,\phi_m): (A,\matholdcal H_X^*)\longrightarrow (A,\matholdcal M)$ in ${_A^{\mathscr C}}Ctr(\psi)$. 

\smallskip
Now let $\mu_X:\matholdcal M(X)\longrightarrow (A,\matholdcal M(X))$ denote the left $A$-module structure on $\matholdcal M(X)$ for each $X\in Ob(\mathscr C)$. By the commutativity of the diagram in \eqref{com3.2}, it may be verified that the maps $\{\mu_X\}_{X\in Ob(\mathscr C)}$ together determine a morphism $\mu:\matholdcal M\longrightarrow 
(A,\matholdcal M)$ in ${_A^{\mathscr C}}Ctr(\psi)$. Additionally, we know that $(u_A,\matholdcal M(X))\circ \mu_X=id$, where $u_A:K\longrightarrow A$ is the unit map of the algebra $A$.  Hence, each $\mu_X$ is a monomorphism. Since kernels in ${_A^{\mathscr C}}Ctr(\psi)$ are computed objectwise, it follows that  $\mu:\matholdcal M\longrightarrow 
(A,\matholdcal M)$ is a monomorphism in ${_A^{\mathscr C}}Ctr(\psi)$. We now form the pullback squares
\begin{equation}\label{321cd2}
\begin{array}{ccc}
\begin{CD}
\matholdcal T @>\phi_m'>> \matholdcal M\\
@V\mu'VV @VV\mu V\\
(A,\matholdcal H_X^*) @>(A,\phi_m)>>  (A,\matholdcal M)\\
\end{CD}&\qquad\qquad &
\begin{CD}
\matholdcal T(X) @>\phi_m'(X)>> \matholdcal M(X)\\
@V\mu'(X)VV @VV\mu(X)=\mu_X V\\
(A,\matholdcal H_X^*(X)) @>(A,\phi_m(X))>>  (A,\matholdcal M(X))\\
\end{CD}\\
\end{array}
\end{equation} The left hand square in \eqref{321cd2} is a pullback in ${_A^{\mathscr C}}Ctr(\psi)$, while the right hand square is a pullback in 
$Vect_K$. Since $\mu$ is a monomorphism, so is $\mu'$. 

\smallskip
We now claim that $m\in \matholdcal M(X)$ lies in the image of $\phi'_m(X):\matholdcal T(X)\longrightarrow \matholdcal M(X)$. In the category of $K$-vector spaces, we note that
$(A,Im(\phi_m(X)))=Im(A,\phi_m(X))$. Since $m\in Im(\phi_m(X))$ and $\matholdcal M(X)$ is a left $A$-module, we can now choose $g\in (A,Im(\phi_m(X)))$ given by setting
$g(a):=am$ for each $a\in A$. Then, $g$ may be treated as an element of $Im(A,\phi_m(X))$. On the other hand, we know by construction that $\mu(X)(m)(a)=am=g(a)$ for each $a\in A$. Hence, $g\in Im(\mu(X))$. Since the right hand square in \eqref{321cd2} is a pullback in $Vect_K$, it follows that $m\in Im(\phi'_m(X))$.

\smallskip
We now consider the collection of subobjects of $(A,\matholdcal H_X^*)$ in  ${_A^{\mathscr C}}Ctr(\psi)$, as $X$ varies over all objects of $\mathscr C$. Since  ${_A^{\mathscr C}}Ctr(\psi)$ is well powered by Lemma \ref{L3.75}, we know that this collection forms a set. We claim that this is a set of generators for  ${_A^{\mathscr C}}Ctr(\psi)$. Indeed, suppose that
$\matholdcal N\hookrightarrow \matholdcal N'$ is a monomorphism in  ${_A^{\mathscr C}}Ctr(\psi)$ that is not an isomorphism. Then, there is some $X\in Ob(\mathscr C)$ such that the inclusion $\matholdcal N(X)\hookrightarrow \matholdcal N'(X)$ is not an isomorphism.  By the above, there is a subobject $\matholdcal T\subseteq (A,\matholdcal H_X^*)$ in 
${_A^{\mathscr C}}Ctr(\psi)$ and a morphism $\matholdcal T\longrightarrow \matholdcal N'$ that does not factor through $\matholdcal N$. This proves the result. 
 \end{proof}
 
 \section{Separability properties for entwined comodules}
 
 In Section 2, we have shown that the forgetful functor $G_\psi :Com_A^{\mathscr C}(\psi)\longrightarrow Com^{\mathscr C}$ has a left adjoint $F_\psi $. In other words, we have natural isomorphisms
 \begin{equation}\label{416g}
 Com_A^{\mathscr C}(\psi)( F_\psi(\mathcal M), \mathcal N)=Com_A^{\mathscr C}(\psi)( \mathcal M\otimes A, \mathcal N)\cong Com^{\mathscr C}( \mathcal M, \mathcal N)
 =Com^{\mathscr C}( \mathcal M, G_\psi(\mathcal N))
\end{equation} for $ \mathcal M\in Com^{\mathscr C}$, $ \mathcal N\in Com^{\mathscr C}_A(\psi)$.  In this section, we will give conditions for the functors $F_\psi $ and 
$G_\psi $ to be separable. 

\subsection{Separability of the functor $F_\psi $} We set $V:=Nat(G_\psi F_\psi ,id_{Com^{\mathscr C}})$.  Motivated by the approach of \cite{uni}, we consider the vector space $V_1$ such that an element $\sigma\in V_1$ is a collection
\begin{equation}\label{4.2z}
\sigma=\{\sigma_X:\mathscr C(X,X)\otimes A\longrightarrow K\}_{X\in Ob(\mathscr C)}
\end{equation}
 such that 
 \begin{equation}\label{4.3z}
 \sigma_Y(f_{Y1}\otimes a_{\psi})f_{Y2}^\psi=\sigma_X(f_{X2}\otimes a)f_{X1} \in \mathscr C(X,Y)
 \end{equation} for any $f\in \mathscr C(X,Y)$, $a\in A$, $X$, $Y\in Ob(\mathscr C)$. 
 
 \begin{lem}\label{Lem4.1}
 There is a linear map $\alpha:V_1\longrightarrow V$ which associates to $\sigma\in V_1$ the natural transformation
 $\tau:=\alpha(\sigma)\in Nat(G_\psi F_\psi ,id_{Com^{\mathscr C}})$ given by setting
 \begin{equation}\label{44r}
 \tau(\mathcal M)(X):\mathcal M(X)\otimes A \longrightarrow \mathcal M(X)\qquad m\otimes a\mapsto m_{X0}\sigma_X(m_{X1}\otimes a)
 \end{equation} for any $\mathcal M\in  Com^{\mathscr C}$  and $X\in Ob(\mathscr C)$. 
 \end{lem}
 \begin{proof}
 We consider $(\mathcal M,\rho^{\mathcal M})\in Com^{\mathscr C}$. First, we will show that $\tau(\mathcal M):\mathcal M\otimes A\longrightarrow \mathcal M$ is a morphism
 in $Com^{\mathscr C}$. In other words, we need to show that the following diagram commutes 
 \begin{equation}\label{cd4.5v}
 \xymatrix{
 \mathcal M(X)\otimes A \ar[rr]^{\rho^{\mathcal M}_{XY}\otimes A\quad}\ar[d]_{\tau(\mathcal M)(X)}&& \mathcal M(Y)\otimes\mathscr C(X,Y)\otimes A\ar[rr]^{\mathcal M(Y)\otimes \psi_{XY}}&& \mathcal M(Y)\otimes A
 \otimes  \mathscr C(X,Y) \ar[d]^{\tau(\mathcal M)(Y)\otimes \mathscr C(X,Y)}\\
 \mathcal M(X) \ar[rrrr]^{\rho^{\mathcal M}_{XY}}&&&&\mathcal M(Y)\otimes \mathscr C(X,Y)\\
 }
 \end{equation}
 for any $X$, $Y\in Ob(\mathscr C)$. For $m\otimes a\in \mathcal M(X)\otimes A$, we check that
 \begin{equation}
 \begin{array}{lll}
(\rho^{\mathcal M}_{XY}\circ \tau(\mathcal M)(X))(m\otimes a)& =\rho^{\mathcal M}_{XY}( m_{X0}\sigma_X(m_{X1}\otimes a))&\\
&=m_{X0Y0}\otimes m_{X0Y1}\sigma_X(m_{X1}\otimes a)&\\
&=m_{Y0}\otimes m_{Y1X1}\sigma_X(m_{Y1X2}\otimes a)&  \mbox{(using \eqref{not2.5})}\\
&=m_{Y0}\otimes m_{Y1Y2}^\psi\sigma_Y(m_{Y1Y1}\otimes a_\psi)& \mbox{(using \eqref{4.3z})}\\
&=m_{Y0Y0}\otimes m_{Y1}^\psi\sigma_Y(m_{Y0Y1}\otimes a_\psi)& \mbox{(using \eqref{not2.5})}\\
&=(\tau(\mathcal M)(Y)\otimes \mathscr C(X,Y))(m_{Y0}\otimes a_\psi\otimes m_{Y1}^\psi)&\\
 \end{array}
 \end{equation} Hence, the diagram \eqref{cd4.5v} commutes. Finally, let $\phi:\mathcal M\longrightarrow \mathcal N$ be a morphism in 
 $Com^{\mathscr C}$. We note that for $m\otimes a\in \mathcal M(X)\otimes A$, we have
 \begin{equation}
 \begin{array}{ll}
( \tau(\mathcal N)\circ (\phi\otimes A))(X)(m\otimes a)&= \tau(\mathcal N)(X)(\phi(X)(m)\otimes a)\\
&=(\phi(X)(m))_{X0}\sigma_X((\phi(X)(m))_{X1}\otimes a)\\
&=\phi(X)(m_{X0})\sigma_X(m_{X1}\otimes a)=(\phi\circ \tau(\mathcal M))(X)(m\otimes a)\\
 \end{array}
 \end{equation} It follows that $\tau=\alpha(\sigma)\in Nat(G_\psi F_\psi ,id_{Com^{\mathscr C}})$. 
 \end{proof}

\begin{lem}\label{Lem4.01}
Let $\mathcal M\in Com^{\mathscr C}$ and $U\in Vect_K$. Let $\tau\in Nat(G_\psi F_\psi ,id_{Com^{\mathscr C}})$ be a natural transformation. Then, we have
\begin{equation}
\tau(U\otimes \mathcal M)(Z) (u\otimes m\otimes a)=u\otimes \tau(\mathcal M)(Z)(m\otimes a)
\end{equation} for any $u\in U$, $m\otimes a\in \mathcal M(Z)\otimes A$, $Z\in Ob(\mathscr C)$.
\end{lem}

\begin{proof}
We choose $u\in U$. We notice that $\phi_u:\mathcal M\longrightarrow U\otimes \mathcal M$ given by setting 
\begin{equation}\phi_u(Z):\mathcal M(Z)\longrightarrow (U\otimes \mathcal M)(Z)\qquad m\mapsto u\otimes m
\end{equation} is a morphism in $Com^{\mathscr C}$.  Since $\tau\in Nat(G_\psi F_\psi ,id_{Com^{\mathscr C}})$ is a natural transformation, we obtain  the following commutative diagram
\begin{equation}
\begin{CD}
\mathcal M(Z)\otimes A @>(\phi_u\otimes A)(Z)>> U\otimes \mathcal M(Z)\otimes A\\
@V\tau(\mathcal M)(Z)VV @VV\mbox{$\tau(U\otimes \mathcal M)(Z)$}V \\
\mathcal M(Z)@>\phi_u(Z)>> U\otimes \mathcal M(Z)\\
\end{CD}
\end{equation} The result is now clear. 
\end{proof}
 
 \begin{lem}\label{Lem4.2}
 There is a linear map $\beta:V\longrightarrow V_1$ which associates to $\tau\in V$ the element $\sigma:=\beta(\tau)\in V_1$ given by setting
 \begin{equation}\label{4fh}
 \sigma_X: \mathscr C(X,X)\otimes A=\mathcal H_X(X)\otimes A\xrightarrow{\quad\tau(\mathcal H_X)(X) \quad} \mathcal H_X(X)=\mathscr C(X,X)\xrightarrow{\epsilon_X}K
 \end{equation} for every $X\in Ob(\mathscr C)$.  
 \end{lem}

\begin{proof}
We have to show that the collection $\{\sigma_X\}_{X\in Ob(\mathscr C)}$ satisfies the condition in \eqref{4.3z}. For this, we consider $f\otimes a\in \mathscr C(X,Y)\otimes A$ for some $X$, $Y\in Ob(\mathscr C)$. Since $\tau(\mathcal H_Y)$ is a morphism in 
$Com^{\mathscr C}$, we have the following commutative diagram
\begin{equation}\label{cd4.9b}
\small
\begin{array}{c}
\xymatrix{
 \mathscr C(X,Y)\otimes A=\mathcal H_Y(X)\otimes A \ar[rr]^{\delta_{XYY}\otimes A}\ar[d]_{\tau(\mathcal H_Y)(X)}&& \mathcal H_Y(Y)\otimes \mathscr C(X,Y)\otimes A
\ar[rr]^{\mathscr C(Y,Y)\otimes \psi_{XY}} &&\mathcal H_Y(Y)\otimes A\otimes \mathscr C(X,Y)\ar[d]^{\tau(\mathcal H_Y)(Y)
\otimes \mathscr C(X,Y)}&&\\
\mathcal H_Y(X)\ar[rrrr]^{\delta_{XYY}} && && \mathcal H_Y(Y)\otimes \mathscr C(X,Y)\ar[rr]^{ \qquad\epsilon_Y\otimes 
\mathscr C(X,Y)}&&\mathscr C(X,Y)\\
}\\
\end{array}
\end{equation} Since the composition of the two bottom horizontal arrows in \eqref{cd4.9b} is the identity, we observe that
\begin{equation}\label{4.10cond}
\tau(\mathcal H_Y)(X)(f\otimes a)=((\epsilon_Y\circ \tau(\mathcal H_Y)(Y))\otimes \mathscr C(X,Y))(f_{Y1}\otimes a_\psi\otimes f_{Y2}^\psi)
=\sigma_Y(f_{Y1}\otimes a_\psi)f_{Y2}^\psi
\end{equation}
We now observe that $\delta^X_Y:\mathcal H_Y\longrightarrow \mathscr C(X,Y)\otimes \mathcal H_X$ defined by setting
\begin{equation}\label{410eqr}
\delta^X_Y(Z):\mathcal H_Y(Z)=\mathscr C(Z,Y)\xrightarrow{\qquad\delta_{ZXY}\qquad} \mathscr C(X,Y)\otimes \mathcal H_X(Z) =\mathscr C(X,Y)\otimes \mathscr C(Z,X)\qquad Z\in Ob(\mathscr C)
\end{equation} is a morphism in $Com^{\mathscr C}$, where $\mathscr C(X,Y)\otimes \mathcal H_X$ is treated as a right $\mathcal C$-comodule in the obvious manner. 
Since $\tau\in Nat(G_\psi F_\psi ,id_{Com^{\mathscr C}})$ is a natural transformation, this gives us the following commutative diagram
\begin{equation}
\label{412cdtr}
\xymatrix{
\mathcal H_Y(X)\otimes A \ar[rr]^{\delta^X_Y(X)\otimes A\quad }\ar[d]_{\tau(\mathcal H_Y)(X)}&& \mathscr C(X,Y)\otimes \mathcal H_X(X) \otimes A
\ar[d]^{\tau(\mathscr C(X,Y)\otimes \mathcal H_X)(X)}&& &&\\ 
\mathcal H_Y(X) \ar[rr]^{\delta^X_Y(X)\quad }&& \mathscr C(X,Y)\otimes \mathcal H_X(X) \ar[rr]^{\mathscr C(X,Y)\otimes \epsilon_X}&& \mathscr C(X,Y)&&\\ 
}
\end{equation}  Since the composition of the two bottom horizontal arrows in \eqref{412cdtr} is the identity, we observe that
\begin{equation}\label{4.11cond}
\begin{array}{lll}
\tau(\mathcal H_Y)(X)(f\otimes a)&= (\mathscr C(X,Y)\otimes \epsilon_X)( \tau(\mathscr C(X,Y)\otimes \mathcal H_X)(X)(f_{X1}\otimes f_{X2}\otimes a))&\\
&= (\mathscr C(X,Y)\otimes \epsilon_X)(f_{X1}\otimes \tau(\mathcal H_X)(X)(f_{X2}\otimes a))&\mbox{(by Lemma \ref{Lem4.01})}\\
&=f_{X1}\sigma_X(f_{X2}\otimes a)&\\
\end{array}
\end{equation} The result is now clear from \eqref{4.10cond} and \eqref{4.11cond}
\end{proof}

\begin{thm}\label{P4.4}
The morphisms $\alpha:V_1\longrightarrow V$ and $\beta: V\longrightarrow V_1$ are mutually inverse isomorphisms. In particular, 
$V_1\cong Nat(G_\psi F_\psi ,id_{Com^{\mathscr C}})$.
\end{thm}

\begin{proof}
We consider $\sigma\in V_1$ and set $\tau: =\alpha(\sigma)$, $\sigma':=\beta(\tau)$. We choose $X\in Ob(\mathscr C)$ and $f\otimes a
\in \mathscr C(X,X)\otimes a$. By the constructions in \eqref{4fh} and \eqref{44r}, we have
\begin{equation}
\begin{array}{ll}
\sigma'_X(f\otimes a)&= \epsilon_X(\tau(\mathcal H_X)(X)(f\otimes a))\\
& = \epsilon_X(f_{X1}\sigma_X(f_{X2}\otimes a))=\sigma_X(\epsilon_X(f_{X1})f_{X2}\otimes a)=\sigma_X(f\otimes a)\\
\end{array}
\end{equation} This shows that $\beta\circ \alpha=id$. 

\smallskip To prove the result, it therefore suffices to show that $\beta$ is a monomorphism. Suppose therefore that we have a natural transformation $0\ne \tau\in Nat(G_\psi F_\psi ,id_{Com^{\mathscr C}})$ such that $\sigma=\beta(\tau)=0$, i.e., $\sigma_X=0$ for each $X\in Ob(\mathscr C)$. In the notation of the proof of Lemma \ref{Lem4.2}, we have the following commutative diagram from \eqref{412cdtr} for $X$, $Y\in Ob(\mathscr C)$
\begin{equation}
\label{412cdtrx}
\xymatrix{
\mathcal H_Y(X)\otimes A \ar[rr]^{\delta^X_Y(X)\otimes A\quad }\ar[d]_{\tau(\mathcal H_Y)(X)}&& \mathscr C(X,Y)\otimes \mathcal H_X(X) \otimes A
\ar[d]^{\tau(\mathscr C(X,Y)\otimes \mathcal H_X)(X)=\mathscr C(X,Y)\otimes \tau(\mathcal H_X)(X)}&& &&\\ 
\mathcal H_Y(X) \ar[rr]^{\delta^X_Y(X)\quad }&& \mathscr C(X,Y)\otimes \mathcal H_X(X) \ar[rr]^{\mathscr C(X,Y)\otimes \epsilon_X}&& \mathscr C(X,Y)&&\\ 
}
\end{equation} where the equality $\tau(\mathscr C(X,Y)\otimes \mathcal H_X)(X)=\mathscr C(X,Y)\otimes \tau(\mathcal H_X)(X)$ in \eqref{412cdtrx} follows from Lemma \ref{Lem4.01}. From \eqref{412cdtrx}, we obtain
\begin{equation}\label{4.19eh}
\tau(\mathcal H_Y)(X) = (\mathscr C(X,Y)\otimes \epsilon_X)\circ \delta^X_Y(X) \circ \tau(\mathcal H_Y)(X)= (\mathscr C(X,Y)\otimes (\epsilon_X\circ \tau(\mathcal H_X)(X)))\circ (\delta^X_Y(X)\otimes A)= (\mathscr C(X,Y)\otimes \sigma_X)\circ (\delta^X_Y(X)\otimes A) =0
\end{equation} for every $X$, $Y\in Ob(\mathscr C)$. By assumption, we know that $\tau\ne 0$. Hence, we can choose some $\mathcal M\in Com^{\mathscr C}$ and $X\in Ob(\mathscr C)$ such that $0\ne \tau(\mathcal M)(X):\mathcal M(X)\otimes A\longrightarrow \mathcal M(X)$. We choose a linear map 
$f: \mathcal M(X)\longrightarrow K$ such that $f\circ  \tau(\mathcal M)(X)\ne 0$. By \eqref{Mc111}, there exists $\phi\in Com^{\mathscr C}(
\mathcal M,\mathcal H_X)\cong Vect_K(\mathcal M(X),K)$ corresponding to $f$ and we have $f=\epsilon_X\circ \phi(X)$. We now have the following commutative diagram
\begin{equation}
\xymatrix{
\mathcal M(X)\otimes A\ar[rr]^{\tau(\mathcal M)(X)} \ar[d]_{\phi(X)\otimes A}&& \mathcal M(X)\ar[d]^{\phi(X)}\ar[rrd]^{f}&& \\
\mathcal H_X(X)\otimes A \ar[rr]_{\tau(\mathcal H_X)(X)=0\quad}&& \mathcal H_X(X)=\mathscr C(X,X)\ar[rr]_{\qquad\quad\epsilon_X}&& K\\
}
\end{equation} Since $f\circ  \tau(\mathcal M)(X)\ne 0$, we have a contradiction. This proves the result. 
\end{proof}

\begin{Thm}\label{T4.5hm}
Let $(\mathscr C,A,\psi)$ be an entwining structure. Then, the functor $F_\psi :Com^{\mathscr C}\longrightarrow Com^{\mathscr C}_A(\psi)$ defined by setting $F_\psi (\mathcal M):=\mathcal M\otimes A$ is separable if and only if there exists $\sigma \in V_1$ such that
\begin{equation}\label{421cond}
\sigma_X(f\otimes 1)=\epsilon_X(f)\qquad \forall \textrm{ }f\in \mathscr C(X,X), \textrm{ }X\in Ob(\mathscr C)
\end{equation}
\end{Thm}

\begin{proof}
From Proposition \ref{P2.6}, it is clear that the unit $\omega: id_{Com^{\mathscr C}}\longrightarrow G_\psi 
F_\psi $ of the adjoint pair $(F_\psi ,G_\psi )$ is given by $\omega(\mathcal M)(X):
\mathcal M(X)\xrightarrow{\quad m\mapsto m\otimes 1\quad }\mathcal M(X)\otimes A$. By \cite[Proposition 1.1]{uni}, we know that the left adjoint 
$F_\psi $ is separable if and only if there exists $\tau\in V=Nat(G_\psi F_\psi ,id_{Com^{\mathscr C}})$ such that
$\tau \circ \omega=id$. 

\smallskip
First, we suppose that there exists $\sigma\in V_1$ satisfying the condition in \eqref{421cond} and set $\tau:=\alpha(\sigma)$. By Lemma \ref{Lem4.1}, we know that 
$\tau(\mathcal M)(X)(m\otimes a)=m_{X0}\sigma_X(m_{X1}\otimes a)$ for any $m\in \mathcal M(X)$, $a\in A$, $\mathcal M\in Com^{\mathscr C}$, 
$X\in Ob(\mathscr C)$. It follows that
\begin{equation}
(\tau(\mathcal M)(X)\circ\omega(\mathcal M)(X))(m)=\tau(\mathcal M)(X)(m\otimes 1)=m_{X_0}\sigma_X(m_{X1}\otimes 1)=m_{X0}\epsilon_X(m_{X1})=m
\end{equation}
Hence, $\tau\circ \omega=id$ and $F_\psi $ is separable.

\smallskip
Conversely, suppose that  $F_\psi $ is separable. Then, there exists a natural transformation $\tau\in V=Nat(G_\psi F_\psi ,id_{Com^{\mathscr C}})$ such that $\tau\circ \omega=id$. In other words, $\tau(\mathcal M)(X)(\omega(\mathcal M)(X)(m))=
\tau(\mathcal M)(X)(m\otimes 1)=m$ for any $m\in \mathcal M(X)$,  $\mathcal M\in Com^{\mathscr C}$, 
$X\in Ob(\mathscr C)$. We set $\sigma:=\beta(\tau)\in V_1$. By the construction in Lemma \ref{Lem4.2}, we now have 
\begin{equation}
\sigma_X(f\otimes 1)=\epsilon_X(\tau(\mathcal H_X)(X)(f\otimes 1))=\epsilon_X(f)
\end{equation} for any $f\in \mathscr C(X,X)$.  This proves the result.

\end{proof}

\subsection{Separability of the functor $G_\psi $}

In a manner similar to Section 4.1, we set $W:=Nat(id_{Com_A^{\mathscr C}(\psi)},F_\psi G_\psi )$. We now consider the space $W_1$ such that an element
$\lambda\in W_1$ is a collection of linear maps
\begin{equation}\label{422lu}
\lambda=\{\mbox{$\lambda^X$ $\vert$ $\lambda^X:\mathscr C(X,X)\longrightarrow A\otimes A$, $\textrm{ }$ $\lambda^{X}(f):= \lambda^{X1}(f)\otimes 
\lambda^{X2}(f)$}\}_{X\in Ob(\mathscr C)}
\end{equation}  (where $\lambda^{X}(f):= \lambda^{X1}(f)\otimes 
\lambda^{X2}(f)$ is written by suppressing summation signs) satisfying the following conditions 
\begin{gather}
\lambda^{Y1}(f_{Y1})\otimes \lambda^{Y2}(f_{Y1})\otimes f_{Y2}=\lambda^{X1}(f_{X2})_\psi\otimes \lambda^{X2}(f_{X2})_\psi\otimes f_{X1}^{\psi\psi} \in A\otimes A\otimes 
\mathscr C(X,Y) \label{425n}\\
\lambda^{Z1}(g)\otimes \lambda^{Z2}(g)a=a_\psi\lambda^{Z1}(g^\psi)\otimes \lambda^{Z2}(g^\psi)\in A\otimes A \label{426n}
\end{gather} for $f\in \mathscr C(X,Y)$, $g\in \mathscr C(Z,Z)$, $a\in A$, $X,Y,Z\in Ob(\mathscr C)$.

 \begin{lem}\label{Lem4.6}
 There is a linear map $\alpha:W_1\longrightarrow W$ which associates to $\lambda\in W_1$ the natural transformation
 $\kappa:=\alpha(\lambda)\in Nat(id_{Com^{\mathscr C}_A(\psi)},F_\psi G_\psi )$ given by setting
 \begin{equation}\label{44rc}
 \kappa(\mathcal M)(X):\mathcal M(X)\longrightarrow \mathcal M(X)\otimes A \qquad m\mapsto m_{X0}\lambda^{X1}(m_{X1})\otimes \lambda^{X2}(m_{X1})
 \end{equation} for any $\mathcal M\in  Com^{\mathscr C}_A(\psi)$  and $X\in Ob(\mathscr C)$. 
 \end{lem}

\begin{proof}
We consider $(\mathcal M,\rho^{\mathcal M},\mu^{\mathcal M})\in  Com^{\mathscr C}_A(\psi)$. We will first show that $\kappa(\mathcal M)$ is a morphism in $ Com^{\mathscr C}_A(\psi)$. To show that $\kappa(\mathcal M)$ is a morphism of $\mathcal C$-comodules, we have to show that the following diagram commutes
\begin{equation*}
\small
\xymatrix{
\mathcal M(X) \ar[rr]^{\rho_{XX}}\ar[dd]^{\rho_{XY}}&&\mathcal M(X)\otimes \mathscr C(X,X) \ar[rr]^{\mathcal M(X)\otimes \lambda^X} &&\mathcal M(X)\otimes A
\otimes A\ar[rr]^{\mu_X\otimes A}&&\mathcal M(X)\otimes A \ar[d]_{\rho_{XY}\otimes A}\\
&& && &&\mathcal M(Y)\otimes \mathscr C(X,Y)\otimes A \ar[d]_{\mathcal M(Y)\otimes \psi_{XY}}\\
\mathcal M(Y)\otimes \mathscr C(X,Y)\ar[rr]^{\rho_{YY}\otimes\mathscr C(X,Y)\quad }&&\mathcal M(Y)\otimes 
\mathscr C(Y,Y)\otimes\mathscr C(X,Y)  \ar[rr]^{\mathcal M(Y)\otimes \lambda^Y\otimes \mathscr C(X,Y)}&&
\mathcal M(Y)\otimes A\otimes A\otimes \mathscr C(X,Y)\ar[rr]^{\mu_Y\otimes A\otimes \mathscr C(X,Y)} && \mathcal M(Y)\otimes A\otimes \mathscr C(X,Y)  \\
}
\end{equation*} For this, we see that we have for $m\in \mathcal M(X)$:

\begin{equation}\label{4.28ar}
\small
\begin{array}{ll}
((\mathcal M(Y)\otimes \psi_{XY})\circ (\rho_{XY}\otimes A)\circ \kappa(\mathcal M)(X))(m) & \\
= ((\mathcal M(Y)\otimes \psi_{XY})\circ (\rho_{XY}\otimes A)\circ (\mu_X\otimes A)\circ (\mathcal M(X)\otimes \lambda^X)\circ \rho_{XX})(m) & \\
= ((\mathcal M(Y)\otimes \psi_{XY})\circ (\mu_Y\otimes \mathscr C(X,Y)\otimes A)\circ (\mathcal M(Y)\otimes \psi_{XY}\otimes A)\circ (\rho_{XY}\otimes A\otimes A)\circ (\mathcal M(X)\otimes \lambda^X)\circ \rho_{XX})(m) & \mbox{(by \eqref{e2.5})}\\
= ((\mu_Y\otimes A\otimes \mathscr C(X,Y))\circ (\mathcal M(Y)\otimes A\otimes \psi_{XY})\circ (\mathcal M(Y)\otimes \psi_{XY}\otimes A)\circ (\rho_{XY}\otimes A\otimes A)\circ (\mathcal M(X)\otimes \lambda^X)\circ \rho_{XX})(m) & \\
= ((\mu_Y\otimes A\otimes \mathscr C(X,Y))\circ (\mathcal M(Y)\otimes A\otimes \psi_{XY})\circ (\mathcal M(Y)\otimes \psi_{XY}\otimes A)\circ (
\mathcal M(Y)\otimes \mathscr C(X,Y)\otimes \lambda^X) \circ 
(\mathcal M(Y)\otimes \delta_{XXY})\circ \rho_{XY})(m) & \\
= m_{Y0}\lambda^{X1}(m_{Y1X2})_\psi\otimes \lambda^{X2}(m_{Y1X2})_\psi\otimes m_{Y1X1}^{\psi\psi}&\\
=m_{Y0}\lambda^{Y1}(m_{Y1Y1})\otimes \lambda^{Y2}(m_{Y1Y1})\otimes m_{Y1Y2}& \mbox{(by \eqref{425n})}\\
= ((\mu_Y\otimes A\otimes \mathscr C(X,Y))\circ (\mathcal M(Y)\otimes \lambda^Y\otimes \mathscr C(X,Y))\circ (\mathcal M(Y)\otimes \delta_{XYY})\circ \rho_{XY})(m)&\\
= ((\mu_Y\otimes A\otimes \mathscr C(X,Y))\circ (\mathcal M(Y)\otimes \lambda^Y\otimes \mathscr C(X,Y))\circ (\rho_{YY}\otimes \mathscr C(X,Y))\circ \rho_{XY})(m)&\\
\end{array}
\end{equation} For $m\otimes a\in \mathcal M(X)\otimes A$, we have
\begin{equation}\label{4.29ry}
(\kappa(\mathcal M)(X)\otimes A)(m\otimes a) 
=m_{X0}\lambda^{X1}(m_{X1})\otimes \lambda^{X2}(m_{X1})a =m_{X0}a_\psi\lambda^{X1}(m_{X1}^\psi)\otimes \lambda^{X2}(m_{X1}^\psi)
= (ma)_{X0}\lambda^{X1}((ma)_{X1})\otimes \lambda^{X2}((ma)_{X1})
\end{equation} where the equality $\lambda^{X1}(m_{X1})\otimes \lambda^{X2}(m_{X1})a =a_\psi\lambda^{X1}(m_{X1}^\psi)\otimes \lambda^{X2}(m_{X1}^\psi)$ follows from \eqref{426n}. It follows from \eqref{4.29ry} that $\kappa(\mathcal M)(X)$ is $A$-linear. Together, it follows from 
\eqref{4.28ar} and \eqref{4.29ry} that $\kappa(\mathcal M)$ is a morphism of entwined comodules.  Finally, it is evident that the maps in 
\eqref{44rc}   are well behaved with respect to morphisms in $Com_A^{\mathscr C}(\psi)$. This shows that we have a natural transformation $\kappa\in Nat(id_{Com^{\mathscr C}_A(\psi)},F_\psi G_\psi )=W$.

\end{proof}

We now set
\begin{equation}\label{430bet}
\beta: W\longrightarrow W_1\qquad\kappa\mapsto \lambda:=\{\lambda^X:\mathscr C(X,X)\xrightarrow{\mathscr C(X,X)\otimes u_A}\mathcal H_X(X)\otimes A
\xrightarrow{\kappa(\mathcal H_X\otimes A)(X)}\mathcal H_X(X)\otimes A\otimes A\xrightarrow{\epsilon_X\otimes A\otimes A} A\otimes A\}_{X\in Ob(\mathscr C)}
\end{equation} where $u_A:K\longrightarrow A$ is the unit map. We need to show that $\lambda=\beta(\kappa)$ satisfies the conditions in
\eqref{425n} and \eqref{426n}.

\begin{lem}\label{Lem4.77}
(a) Let $\kappa\in W=Nat(id_{Com_A^{\mathscr C}(\psi)},F_\psi G_\psi )$ and $\lambda:=\beta(\kappa)$. Then, for each $X\in Ob(\mathscr C)$, we have
\begin{equation}\label{431cx}
\lambda^X=(\epsilon_X\otimes A\otimes A)\circ (\kappa(\mathcal H_X\otimes A)(X))\circ(\mathscr C(X,X)\otimes u_A)=(A\otimes \epsilon_X\otimes A)\circ (\kappa(A\otimes \mathcal H_X)(X))\circ(u_A\otimes \mathscr C(X,X))
\end{equation}

(b) Let $\mathcal M\in Com^{\mathscr C}_A(\psi)$ and $U\in Vect_K$. Let $\kappa\in Nat(id_{Com_A^{\mathscr C}(\psi)},F_\psi G_\psi )$ be a natural transformation. Then, we have
\begin{equation}
\kappa(U\otimes \mathcal M)(Z) (u\otimes m)=u\otimes \kappa(\mathcal M)(Z)(m)
\end{equation} for any $u\in U$, $m\in \mathcal M(Z) $, $Z\in Ob(\mathscr C)$.

\end{lem}

\begin{proof}
(a) By Lemma \ref{L2.2} and Lemma \ref{L2.3}, we know that both $\mathcal H_X\otimes A$, $A\otimes \mathcal H_X\in Com^{\mathscr C}_A(\psi)$. It may be easily verified that
the maps $\psi_{YX}:\mathscr C(Y,X)\otimes A\longrightarrow A\otimes \mathscr C(Y,X)$ for $Y\in Ob(\mathscr C)$ determine a morphism $\psi_X:\mathcal H_X\otimes A
\longrightarrow A\otimes \mathcal H_X$ in $Com^{\mathscr C}_A(\psi)$. Since $\kappa\in Nat(id_{Com_A^{\mathscr C}(\psi)},F_\psi G_\psi )$, we now have a commutative diagram
\begin{equation}
\begin{CD}
\mathscr C(X,X)@>\mathscr C(X,X)\otimes u_A>> \mathcal H_X(X)\otimes A @>\kappa(\mathcal H_X\otimes A)(X)>> \mathcal H_X(X)\otimes A\otimes A @>\epsilon_X\otimes A\otimes A>> A\otimes A\\
@VidVV @V\psi_{X}(X)=\psi_{XX} VV @V\psi_X(X)\otimes AV=\psi_{XX}\otimes AV  @VidVV\\
\mathscr C(X,X)@>u_A\otimes \mathscr C(X,X)>>A\otimes \mathcal H_X(X) @>\kappa(A\otimes \mathcal H_X)(X)>> A\otimes \mathcal H_X(X)\otimes A @>A\otimes\epsilon_X\otimes A>> A\otimes A  \\
\end{CD}
\end{equation}  The result is now clear. 

\smallskip
(b) This is immediate from the fact that for each $u\in U$, the maps $\phi_u(Z):\mathcal M(Z)\xrightarrow{m\mapsto u\otimes m} U\otimes\mathcal M(Z)$ define a morphism
$\phi_u$ in $Com_A^{\mathscr C}(\psi)$. 
\end{proof}

\begin{lem}\label{Lem4.87}
For $\kappa\in W$, the collection $\lambda$ as defined in \eqref{430bet} satisfies
\begin{gather}
\lambda^{Y1}(f_{Y1})\otimes \lambda^{Y2}(f_{Y1})\otimes f_{Y2}=\lambda^{X1}(f_{X2})_\psi\otimes \lambda^{X2}(f_{X2})_\psi\otimes f_{X1}^{\psi\psi} \in A\otimes A\otimes 
\mathscr C(X,Y) \label{433n}\\
\lambda^{Z1}(g)\otimes \lambda^{Z2}(g)a=a_\psi\lambda^{Z1}(g^\psi)\otimes \lambda^{Z2}(g^\psi)\in A\otimes A \label{434n}
\end{gather} for $f\in \mathscr C(X,Y)$, $g\in \mathscr C(Z,Z)$, $a\in A$, $X,Y,Z\in Ob(\mathscr C)$. It follows that we have a well defined map 
$\beta:W\longrightarrow W_1$. 
\end{lem}

\begin{proof}
Since $\kappa(\mathcal H_Y\otimes A)$ is  a morphism in $Com_A^{\mathscr C}(\psi)$ and hence in $Com^{\mathscr C}$, we have the following commutative diagram
\begin{equation}\label{435cd}
\begin{CD}
\mathscr C(X,Y)\otimes A @>\kappa(\mathcal H_Y\otimes A)(X)>>  \mathscr C(X,Y)\otimes A \otimes A  @>(A\otimes \psi_{XY})\circ (\psi_{XY}\otimes A)>> A \otimes A\otimes \mathscr C(X,Y)\\
@V(\mathscr C(Y,Y)\otimes \psi_{XY})\circ(\delta_{XYY}\otimes A)VV @VV(\mathscr C(Y,Y)\otimes A\otimes \psi_{XY})\circ ((\mathscr C(Y,Y)\otimes \psi_{XY})\circ\delta_{XYY})\otimes A)V @VVidV \\
\mathscr C(Y,Y)\otimes A \otimes \mathscr C(X,Y)@>\kappa(\mathcal H_Y\otimes A)(Y)\otimes 
\mathscr C(X,Y)>>  \mathscr C(Y,Y)\otimes A \otimes A\otimes \mathscr C(X,Y) @>\epsilon_Y\otimes A \otimes A\otimes \mathscr C(X,Y)>>A \otimes A\otimes \mathscr C(X,Y)  \\
\end{CD}
\end{equation} From \eqref{435cd} it follows that for $f\in \mathscr C(X,Y)$ we have
\begin{equation}\label{436we}
((A\otimes \psi_{XY})\circ (\psi_{XY}\otimes A)\circ (\kappa(\mathcal H_Y\otimes A)(X)))(f\otimes 1)=\lambda^{Y1}(f_{Y1})\otimes \lambda^{Y2}(f_{Y1})\otimes f_{Y2}
\end{equation} We now recall the morphisms $\delta^X_Y:\mathcal H_Y\longrightarrow \mathscr C(X,Y)\otimes \mathcal H_X$ in $Com^{\mathscr C}$ defined in \eqref{410eqr}. Accordingly, 
$\delta^X_Y\otimes A:\mathcal H_Y\otimes A\longrightarrow \mathscr C(X,Y)\otimes \mathcal H_X\otimes A$ is a morphism in $Com^{\mathscr C}_A(\psi)$. The natural transformation 
$\kappa$ now gives a commutative diagram
\begin{equation*}
\small
\xymatrix{
\mathscr C(X,Y)\otimes A \ar[rrr]^{\kappa(\mathcal H_Y\otimes A)(X)}\ar[d]_{(\delta^X_Y\otimes A)(X)}&&& \mathscr C(X,Y)\otimes A\otimes A \ar[d]_{(\delta^X_Y\otimes A)(X)\otimes A}\ar[rrd]^{id} &&  && \\
\mathscr C(X,Y)\otimes \mathscr C(X,X)\otimes A \ar[rrr]^{\kappa( \mathscr C(X,Y)\otimes \mathcal H_X\otimes A)(X)}_{=\mathscr C(X,Y)\otimes \kappa( \mathcal H_X\otimes A)(X)}&&& \mathscr C(X,Y)\otimes \mathscr C(X,X) \otimes A\otimes A \ar[rr]^{\qquad\mathscr C(X,Y)\otimes \epsilon_X\otimes A\otimes A}&& \mathscr C(X,Y)
\otimes A\otimes A\ar[rr]^{(A\otimes \psi_{XY})\circ (\psi_{XY}\otimes A)} && A\otimes A\otimes \mathscr C(X,Y) \\
}
\end{equation*} where the equality $\kappa( \mathscr C(X,Y)\otimes \mathcal H_X\otimes A)(X)=\mathscr C(X,Y)\otimes \kappa( \mathcal H_X\otimes A)(X)$ follows from 
Lemma \ref{Lem4.77}(b).  It follows that for $f\in \mathscr C(X,Y)$ we have
\begin{equation}\label{438we}
((A\otimes \psi_{XY})\circ (\psi_{XY}\otimes A)\circ (\kappa(\mathcal H_Y\otimes A)(X)))(f\otimes 1)=\lambda^{X1}(f_{X2})_\psi\otimes \lambda^{X2}(f_{X2})_\psi\otimes f_{X1}^{\psi\psi} 
\end{equation} From \eqref{436we} and \eqref{438we}, we see that the condition \eqref{433n} is satisfied by $f\in \mathscr C(X,Y)$. On the other hand, since $\kappa(A\otimes \mathcal H_Z)(Z)$ is right $A$-linear for $Z\in Ob(\mathscr C)$, we have the commutative diagram
\begin{equation}\label{439cdf}
\begin{CD}
A\otimes \mathscr C(Z,Z) \otimes A@>\kappa(A\otimes \mathcal H_Z)(Z)\otimes A>> A\otimes \mathscr C(Z,Z)\otimes A\otimes A@>A\otimes\epsilon_Z\otimes A\otimes A>>A\otimes A\otimes A\\
@V(\mu_A\otimes \mathscr C(Z,Z))\circ (A\otimes\psi_{ZZ})VV @VVA\otimes \mathscr C(Z,Z)\otimes \mu_AV @VVA\otimes \mu_AV\\ 
A\otimes \mathscr C(Z,Z)@>\kappa(A\otimes \mathcal H_Z)(Z)>> A\otimes \mathscr C(Z,Z)\otimes A  @>A\otimes\epsilon_Z\otimes A>> A\otimes A\\
\end{CD}
\end{equation} where $\mu_A:A\otimes A\longrightarrow A$ is the multiplication on $A$. From \eqref{439cdf} and \eqref{431cx}, it follows that for $a\in A$ and 
$g\in \mathscr C(Z,Z)$, we have
\begin{equation}\label{440ve}
((A\otimes\epsilon_Z\otimes A)\circ (\kappa(A\otimes \mathcal H_Z)(Z)))(a_\psi\otimes g^\psi)=\lambda^{Z1}(g)\otimes \lambda^{Z2}(g)a
\end{equation} It may be easily verified that the maps
\begin{equation}
A\otimes\mathcal H_Z(Z')\otimes A=A\otimes \mathscr C(Z',Z)\otimes A\xrightarrow{A\otimes \psi_{Z'Z}}A\otimes A\otimes \mathscr C(Z',Z)\xrightarrow{\mu_A
\otimes \mathscr C(Z',Z)}A\otimes \mathcal H_Z(Z')\qquad Z'\in Ob(\mathscr C)
\end{equation}   define a morphism $A\otimes \mathcal H_Z\otimes A\longrightarrow A\otimes \mathcal H_Z$ in $Com^{\mathscr C}_A(\psi)$. Accordingly, the natural transformation $\kappa$ gives us a commutative diagram
\begin{equation}\label{442cdu}
\begin{CD}
\mathscr C(Z,Z)\otimes A\otimes A @>\psi_{ZZ}\otimes A>> A\otimes \mathscr C(Z,Z)\otimes A@>\kappa(A\otimes \mathcal H_Z\otimes A)(Z)>=A\otimes \kappa(\mathcal H_Z\otimes A)(Z)> A\otimes \mathscr C(Z,Z)\otimes A\otimes A @>A\otimes \epsilon_Z\otimes A\otimes A>>A\otimes A\otimes A\\ 
@. @V(\mu_A\otimes \mathscr C(Z,Z))\circ (A\otimes\psi_{ZZ})VV @VV((\mu_A\otimes \mathscr C(Z,Z))\circ (A\otimes\psi_{ZZ}))\otimes AV @VV\mu_A\otimes AV\\
@. A\otimes \mathscr C(Z,Z) @>\kappa(A\otimes \mathcal H_Z)(Z)>>A\otimes \mathscr C(Z,Z)\otimes A@>A\otimes \epsilon_Z\otimes A>>A\otimes A \\
\end{CD}
\end{equation} where  the equality $\kappa(A\otimes \mathcal H_Z\otimes A)(Z)=A\otimes \kappa(\mathcal H_Z\otimes A)(Z)$ follows from 
Lemma \ref{Lem4.77}(b). Taking $g\otimes a\otimes 1\in \mathscr C(Z,Z)\otimes A\otimes A $, it follows from \eqref{442cdu}  that
\begin{equation}\label{442cde}
((A\otimes\epsilon_Z\otimes A)\circ (\kappa(A\otimes \mathcal H_Z)(Z)))(a_\psi\otimes g^\psi)=a_\psi\lambda^{Z1}(g^\psi)\otimes \lambda^{Z2}(g^\psi)
\end{equation} From \eqref{440ve} and \eqref{442cdu}, it is clear that the condition \eqref{434n} is satisfied for $g\in \mathscr C(Z,Z)$ and $a\in A$. This proves the result.
\end{proof}

\begin{lem}\label{LRG4}
The map $\beta:W\longrightarrow W_1$ is a monomorphism. 
\end{lem}

\begin{proof}
Suppose we have
$0\ne\kappa\in W$ such that $\lambda:=\beta(\kappa)=0$, i.e., $\lambda^X=0$ for each $X\in Ob(\mathscr C)$. From the proof of Lemma \ref{Lem4.87}, we have the commutative diagram
\begin{equation}\label{445xym}
\xymatrix{
 \mathscr C(X,Y) \ar[d]_{\mathscr C(X,Y)\otimes u_A}&&& \\
\mathscr C(X,Y)\otimes A \ar[rrr]^{\kappa(\mathcal H_Y\otimes A)(X)}\ar[d]_{(\delta^X_Y\otimes A)(X)}&&& \mathscr C(X,Y)\otimes A\otimes A \ar[d]_{(\delta^X_Y\otimes A)(X)\otimes A}\ar[rrd]^{id} \\
\mathscr C(X,Y)\otimes \mathscr C(X,X)\otimes A \ar[rrr]^{\kappa( \mathscr C(X,Y)\otimes \mathcal H_X\otimes A)(X)}_{=\mathscr C(X,Y)\otimes \kappa( \mathcal H_X\otimes A)(X)}&&& \mathscr C(X,Y)\otimes \mathscr C(X,X) \otimes A\otimes A \ar[rr]^{\qquad\mathscr C(X,Y)\otimes \epsilon_X\otimes A\otimes A}&& \mathscr C(X,Y)
\otimes A\otimes A  \\
}
\end{equation} From \eqref{445xym} it follows that for any $f\in \mathscr C(X,Y)$, we have $\kappa(\mathcal H_Y\otimes A)(X)(f\otimes 1)=f_{X1}\otimes \lambda^X(f_{X2})=0$. Additionally since $\kappa(\mathcal H_Y\otimes A)(X)$ is right $A$-linear, it follows that $\kappa(\mathcal H_Y\otimes A)(X)(f\otimes a)=0$ for any $a\in A$. Hence, $\kappa(\mathcal H_Y\otimes A)(X)=0$ for $X$, $Y\in Ob(\mathscr C)$. 

\smallskip
Suppose there exists  $\mathcal N\in Com^{\mathscr C}$ such that  $\kappa(\mathcal N\otimes A)\ne 0$. Then, there is $X\in Ob(\mathscr C)$ such that $0\ne \kappa(\mathcal N\otimes A)(X): \mathcal N(X)\otimes A\longrightarrow \mathcal N(X)\otimes A\otimes A$. So we can choose a linear map $h:\mathcal N(X)\longrightarrow K$ such that $(h\otimes 
A\otimes A)\circ  \kappa(\mathcal N\otimes A)(X)
\ne 0$.  By \eqref{Mc111}, we have $\phi\in Com^{\mathscr C}(
\mathcal N,\mathcal H_X)\cong Vect_K(\mathcal N(X),K)$ corresponding to $h$ and we have $h=\epsilon_X\circ \phi(X)$. This gives us the following commutative diagram
\begin{equation}\label{445fv}
\xymatrix{
\mathcal N(X)\otimes A\ar[rr]^{\kappa(\mathcal N\otimes A)(X)} \ar[d]_{\phi(X)\otimes A}&& \mathcal N(X)\otimes A\otimes A\ar[d]^{\phi(X)\otimes A\otimes A}\ar[rrd]^{h\otimes A\otimes A}&& \\
\mathcal H_X(X)\otimes A \ar[rr]_{\kappa(\mathcal H_X\otimes A)(X)=0\qquad\qquad\quad}&& \mathcal H_X(X)\otimes A\otimes A=\mathscr C(X,X)\otimes A\otimes A\ar[rr]_{\qquad\quad\qquad \epsilon_X\otimes A\otimes A}&&  A\otimes A\\
}
\end{equation} From \eqref{445fv}, we get $(h\otimes 
A\otimes A)\circ  \kappa(\mathcal N\otimes A)(X)
= 0$ which is a contradiction. It follows that $\kappa(\mathcal N\otimes A)=0$ for every $\mathcal N\in Com^{\mathscr C}$. 

\smallskip
Finally, we consider $(\mathcal M,\rho^{\mathcal M},\mu^{\mathcal M})\in Com^{\mathscr C}_A(\psi)$. By treating $\mathcal 
M$ as an object of $Com^{\mathscr C}$, we have $\mathcal M\otimes A\in Com^{\mathscr C}_A(\psi)$ as in Lemma \ref{L2.2}. It may be easily verified
that the maps $\nu(X):\mathcal M(X)\otimes A\xrightarrow{\mu_X}\mathcal M(X)$ induce a morphism in $Com^{\mathscr C}_A(\psi)$. The natural transformation 
$\kappa$ now gives a commutative diagram
\begin{equation}\label{446cd}
\begin{CD}
\mathcal M(X)\otimes A @>\kappa(\mathcal M\otimes A)(X)=0>> \mathcal M(X)\otimes A\otimes A\\
@V\nu(X)VV @VV\nu(X)\otimes AV \\
\mathcal M(X) @>\kappa(\mathcal M)(X)>> \mathcal M(X)\otimes A\\
\end{CD}
\end{equation} Since each $\nu(X):\mathcal M(X)\otimes A\xrightarrow{\mu_X}\mathcal M(X)$ is an epimorphism, it is now clear that
$\kappa(\mathcal M)(X)=0$. This proves the result.

\end{proof}

\begin{thm}\label{P4.9}
The morphisms $\alpha:W_1\longrightarrow W$ and $\beta: W\longrightarrow W_1$ are mutually inverse isomorphisms. In particular, 
$W_1\cong  Nat(id_{Com_A^{\mathscr C}(\psi)},F_\psi G_\psi )$.
\end{thm}

\begin{proof}
Let $\lambda\in W_1$. We set $\kappa:=\alpha(\lambda)$ and $\lambda':=\beta(\kappa)$. For $X\in Ob(\mathscr C)$ and $f\in \mathscr C(X,X)$, we have
\begin{equation}\label{4.44rt}
\begin{array}{ll}
\lambda'^X(f) &=(\epsilon_X\otimes A\otimes A)( \kappa(\mathcal H_X\otimes A)(X)(f\otimes 1))\\
&=(\epsilon_X\otimes A\otimes A)(f_{X1}\otimes \lambda^{X1}(f_{X2})\otimes \lambda^{X2}(f_{X2}))\\
&=\epsilon_X(f_{X1})\lambda^{X1}(f_{X2})\otimes \lambda^{X2}(f_{X2})=\epsilon_X(f_{X1})\lambda^{X}(f_{X2}) =\lambda^X(f)\\
\end{array}
\end{equation} Hence, we see that $\beta\circ\alpha=id$. By Lemma \ref{LRG4}, we know that $\beta$ is a monomorphism, and the result follows. 

\end{proof}

\begin{Thm}\label{T4.11hm}
Let $(\mathscr C,A,\psi)$ be an entwining structure. Then, the functor $G_\psi :Com^{\mathscr C}_A(\psi)\longrightarrow Com^{\mathscr C}$   is separable if and only if there exists $\lambda \in W_1$ such that
\begin{equation}\label{448cond}
\lambda^{X1}(f)\lambda^{X2}(f)=\epsilon_X(f)\cdot 1 \in A\qquad \forall \textrm{ }f\in \mathscr C(X,X), \textrm{ }X\in Ob(\mathscr C)
\end{equation}
\end{Thm}

\begin{proof}
We see that the counit $\omega':F_\psi G_\psi \longrightarrow id_{Com_A^{\mathscr C}(\psi)}$  of the adjunction $(F_\psi ,G_\psi )$ is given by $\omega'(\mathcal M)(X):\mathcal M(X)\otimes A\xrightarrow{\textrm{ } m\otimes a\mapsto ma\textrm{ }}\mathcal M(X)$. Applying \cite[Proposition 1.1]{uni}, we know that the right
adjoint $G_\psi $ is separable if and only if there exists $\kappa\in W$ such that $\omega'\circ\kappa=id$. 

\smallskip
First, we suppose there exists $\lambda\in W_1$ satisfying the condition in \eqref{448cond}. We set $\kappa:=\alpha(\lambda)$. For $\mathcal M
\in Com_A^{\mathscr C}(\psi)$, $X\in Ob(\mathscr C)$ and $m\in \mathcal M(X)$, we see that
\begin{equation}
(\omega'(\mathcal M)(X)\circ \kappa(\mathcal M)(X))(m)=\omega'(\mathcal M)(X)(m_{X0}\lambda^{X1}(m_{X1})\otimes \lambda^{X2}(m_{X1}))=m_{X0}\lambda^{X1}(m_{X1}) \lambda^{X2}(m_{X1})=m_{X0}\epsilon_X(m_{X1})=m
\end{equation} Hence,  $\omega'\circ\kappa=id$ and $G_\psi $ is separable.

\smallskip
Conversely, suppose that $G_\psi $ is separable, i.e., there exists a natural transformation $\kappa\in W$ such that  $\omega'\circ\kappa=id$. We set $
\lambda:=\beta(\kappa)$. From the definition in \eqref{430bet}, we now see that for $f\in \mathscr C(X,X)$, we have
\begin{equation}
\lambda^{X1}(f)\lambda^{X2}(f)=((\epsilon_X\otimes A)\circ (\omega'(\mathcal H_X\otimes A)(X))\circ (\kappa(\mathcal H_X\otimes A)(X))\circ (\mathscr C(X,X)\otimes u_A) )(f)=\epsilon_X(f)\cdot 1
\end{equation} This proves the result.
\end{proof}

 \section{Separability properties for entwined contramodules}
We continue with the entwining structure $(\mathscr C,A,\psi)$. In Proposition \ref{P3.7}, we have shown that the forgetful functor  $ S_\psi: {_A^{\mathscr C}}Ctr(\psi)\longrightarrow {^\mathscr C}Ctr$ has a right adjoint $T_\psi$. In other words, we have natural isomorphisms
\begin{equation}\label{5.1iso}
{^\mathscr C}Ctr(S_\psi(\matholdcal M),\matholdcal N)={^\mathscr C}Ctr(\matholdcal M,\matholdcal N)\cong {_A^{\mathscr C}}Ctr(\psi)(\matholdcal M,(A,\matholdcal N))={_A^{\mathscr C}}Ctr(\psi)(\matholdcal M,T_\psi(\matholdcal N))
\end{equation} for $\matholdcal M\in {_A^{\mathscr C}}Ctr(\psi)$ and $\matholdcal N\in {^\mathscr C}Ctr$. In this section, we will give conditions for the functors $S_\psi$ and $T_\psi$ to be separable.

\smallskip

 In order to work with contramodules, we will use the notation $ (L_k,L_{k-1},...,L_1):=(L_k,(L_{k-1},(...(L_2,L_1))))$ defined in \eqref{ntn3}.  We will often need to replace a vector space $U$ by the corresponding representable functor $\widetilde{U}:Vect_K\xrightarrow{\quad (U,\_\_)\quad} Vect_K$. If $f:U\longrightarrow U'$ is a morphism in $Vect_K$, we will denote by $\widetilde{f}:\widetilde{U'}\longrightarrow \widetilde{U}$ the induced morphism of representable functors. If $U$, $U'\in Vect_K$, we notice that the composition 
 $\widetilde{U'}\ast\widetilde{U}:=(U'(U,\_\_))=(U\otimes U',\_\_)=\widetilde{U\otimes U'}$. For instance, $\delta_{XYZ}:\mathscr C(X,Z)\longrightarrow \mathscr C(Y,Z)\otimes \mathscr C(X,Y)$ will often be expressed as 
 \begin{equation}
 \widetilde{\delta}_{XYZ}:\widetilde{\mathscr C(X,Y)}\ast \widetilde{\mathscr C(Y,Z)}\longrightarrow \widetilde{\mathscr C(X,Z)}
 \end{equation} Similarly, the morphisms $\psi_{XY}$ in the entwining structure $\psi$ will often  be replaced by $\widetilde{\psi}_{XY}:\widetilde{\mathscr C(X,Y)}
 \ast \widetilde{A}\longrightarrow \widetilde{A}\ast \widetilde{\mathscr C(X,Y)}$. An element $\sigma\in V_1$ as defined in \eqref{4.2z} can now be expressed as a collection
 of natural transformations
 \begin{equation}\label{5.2z}
\sigma=\{\widetilde{\sigma}_X:id_{Vect_K}\longrightarrow \widetilde{A}\ast \widetilde{\mathscr C(X,X)} \}_{X\in Ob(\mathscr C)}
\end{equation}
 such that 
\begin{equation}\label{5.3zz}
\begin{CD}
(\mathscr C(X,Y),\_\_) =\widetilde{\mathscr C(X,Y)}@>(\widetilde A\ast \widetilde\delta_{XYY})\circ (\widetilde{\psi}_{XY}\ast \widetilde{\mathscr C(Y,Y)})\circ (\widetilde{\mathscr C(X,Y)}\ast \widetilde\sigma_Y)\textrm{ }=\textrm{ }
 (\widetilde A\ast\widetilde\delta_{XXY})\circ (\widetilde\sigma_X\ast \widetilde{\mathscr C(X,Y)})>
(A,\delta_{XYY},\_\_)\circ (\psi_{XY},\mathscr C(Y,Y),\_\_)\circ (\mathscr C(X,Y), \sigma_Y,\_\_)= (A,\delta_{XXY},\_\_)\circ ( \sigma_X,\mathscr C(X,Y),\_\_)> \widetilde A\ast  \widetilde{\mathscr C(X,Y)}=(A,(\mathscr C(X,Y),\_\_))\\
 \end{CD}
\end{equation} We also set $V':=Nat(id_{{^\mathscr C}Ctr},S_\psi T_\psi)$. We will now show that $V'$ and $V_1$ are isomorphic.

\begin{lem}\label{L5.1}

 There is a linear map $\alpha:V_1\longrightarrow V'$ which associates to $\sigma\in V_1$ the natural transformation
 $\tau:=\alpha(\sigma)\in Nat(id_{{^\mathscr C}Ctr},S_\psi T_\psi)$ given by setting
 \begin{equation}\label{44r5}
 \tau(\matholdcal M)(X):\matholdcal M(X) \xrightarrow{\widetilde\sigma_X(\matholdcal M(X))=(\sigma_X,\matholdcal M(X))} (\widetilde{A}\ast \widetilde{\mathscr C(X,X)})(\matholdcal M(X))=(A,\mathscr C(X,X),\matholdcal M(X))
 \xrightarrow{(A,\pi_{XX})=\widetilde A(\pi_{XX})}   (A, \matholdcal M(X)) 
 \end{equation} for any $(\matholdcal M,\pi^{\matholdcal M}) \in   {^\mathscr C}Ctr$  and $X\in Ob(\mathscr C)$. 
 \end{lem}
 
 \begin{proof}
 We start by showing that $\tau(\matholdcal M):\matholdcal M\longrightarrow (A,\matholdcal M)$ is indeed a morphism in ${^\mathscr C}Ctr$. In other words, we have to show that the 
 following diagram commutes for $X$, $Y\in Ob(\mathscr C)$:
 \begin{equation}
 \xymatrix{
 (\mathscr C(X,Y),\matholdcal M(Y)) \ar[rrr]^{( \widetilde{\mathscr C(X,Y)}\ast \widetilde\sigma_Y)(\matholdcal M(Y))\qquad}_{=(\mathscr C(X,Y),\sigma_Y,\matholdcal M(Y))\qquad}\ar[dd]_{\pi_{XY}}&&& (\mathscr C(X,Y),A,\mathscr C(Y,Y),\matholdcal M(Y))\ar[rrr]^{(\mathscr C(X,Y),A,\pi_{YY})}\ar[d]^{(\psi_{XY},\mathscr C(Y,Y),\matholdcal M(Y))}
&&& (\mathscr C(X,Y),A,\matholdcal M(Y)) \ar[d]^{(\psi_{XY},\matholdcal M(Y))}\\
 &&& (A,\mathscr C(X,Y),\mathscr C(Y,Y),\matholdcal M(Y)) \ar[rrr]^{(A,\mathscr C(X,Y),\pi_{YY})} &&& (A,\mathscr C(X,Y),\matholdcal M(Y))\ar[d]^{(A,\pi_{XY})}\\
 \matholdcal M(X) \ar[rrr]^{\widetilde\sigma_X(\matholdcal M(X))}_{=(\sigma_X,\matholdcal M(X))}&&& (A,\mathscr C(X,X),\matholdcal M(X))\ar[rrr]^{(A,\pi_{XX})}&&& (A,\matholdcal M(X))\\
 }
 \end{equation} For this, we see that
 \begin{equation}
 \begin{array}{lr}
 (A,\pi_{XY})\circ (\psi_{XY},\matholdcal M(Y))\circ (\mathscr C(X,Y),A,\pi_{YY})\circ (\mathscr C(X,Y),\sigma_Y,\matholdcal M(Y)&\\
 = (A,\pi_{XY})\circ (A,\mathscr C(X,Y),\pi_{YY}) \circ (\psi_{XY},\mathscr C(Y,Y),\matholdcal M(Y)) \circ (\mathscr C(X,Y),\sigma_Y,\matholdcal M(Y))&\\
  = (A,\pi_{XY})\circ (A,\delta_{XYY},\matholdcal M(Y)) \circ (\psi_{XY},\mathscr C(Y,Y),\matholdcal M(Y)) \circ (\mathscr C(X,Y),\sigma_Y,\matholdcal M(Y))&
  \mbox{(by \eqref{dg3.4})}\\
   = (A,\pi_{XY})\circ (A,\delta_{XXY},\matholdcal M(Y)) \circ  (\sigma_X,\mathscr C(X,Y),\matholdcal M(Y))&
  \mbox{(by \eqref{5.3zz})}\\
   =(A,\pi_{XX})\circ (A,\mathscr C(X,X),\pi_{XY})\circ  \widetilde\sigma_X(\mathscr C(X,Y),\matholdcal M(Y))  &  \mbox{(by \eqref{dg3.4})}\\
 =(A,\pi_{XX})\circ \widetilde\sigma_X(\matholdcal M(X)) \circ \pi_{XY}=(A,\pi_{XX})\circ (\sigma_X,\matholdcal M(X))\circ \pi_{XY}&\\
 \end{array}
 \end{equation} If $\phi:\matholdcal M\longrightarrow \matholdcal N$ is a morphism in ${^\mathscr C}Ctr$, it is also clear that the maps in \eqref{44r5} are well behaved with 
 respect to the maps $\phi(X):\matholdcal M(X)\longrightarrow \matholdcal N(X)$ for $X\in Ob(\mathscr C)$. This proves the result. 
 \end{proof}

 Now suppose we have $\tau\in V'=Nat(id_{{^\mathscr C}Ctr},S_\psi T_\psi)$. For each $X\in Ob(\mathscr C)$, we can define a natural transformation  by setting for $U\in Vect_K$:
 \begin{equation}\label{5.8uu}
 \begin{CD}
 \widetilde\sigma_X(U): U @>(\epsilon_X,U)>> (\mathscr C(X,X),U) =\matholdcal H_X^U(X)@>\tau(\mathcal H_X^U)(X)>> (A,\matholdcal H_X^U)(X)=(A,\mathscr C(X,X),U)  
 \end{CD}
 \end{equation} where $\matholdcal H_X^U\in {^\mathscr C}Ctr$ as in \eqref{33tr}. Then \eqref{5.8uu} corresponds to a map $\sigma_X:\mathscr C(X,X)\otimes A\longrightarrow K$. Since $\tau$ is a natural transformation, we have a collection of maps 
 \begin{equation}\label{5yon}
 \tau(\matholdcal H_Y^U)(X): (\mathscr C(X,Y),U)\longrightarrow (A,\mathscr C(X,Y),U)
 \end{equation} that are functorial in $U$. By Yoneda lemma, it follows from \eqref{5yon} that we have maps
 $\zeta^\tau_{XY}:\mathscr C(X,Y)\otimes A\longrightarrow \mathscr C(X,Y)$ such that $ \tau(\matholdcal H_Y^U)(X)=(\zeta^\tau_{XY},U)$ for $U\in Vect_K$. From \eqref{5.8uu}, it is now clear that $\sigma_X=\epsilon_X\circ \zeta^\tau_{XX}$.

 \begin{lem}\label{L5.01b} For $L$,  $U\in Vect_K$, we have
 \begin{equation}\label{5.10eqk}
 \tau(\matholdcal H_Y^{(L,U)})(X)=(L\otimes \zeta^\tau_{XY},U): (L\otimes\mathscr C(X,Y),U)=(\mathscr C(X,Y),L,U) \longrightarrow (A,\mathscr C(X,Y),L,U) =(L\otimes \mathscr C(X,Y)\otimes A,U)
 \end{equation}
 
 \end{lem}
 
 \begin{proof}
It is clear that an  element $w\in L$ induces a morphism $w:(L,U)\longrightarrow U$ and hence a morphism $\matholdcal H_Y^{(L,U)}\longrightarrow \matholdcal H_Y^U$ in  ${^\mathscr C}Ctr$. 
The natural transformation $\tau$ now gives us a commutative diagram
\begin{equation}\label{510cdp}
\begin{CD}
\matholdcal H_Y^{(L,U)}(X)=(\mathscr C(X,Y),L,U) @>\tau(\matholdcal H_Y^{(L,U)})(X)>> (A,\matholdcal H_Y^{(L,U)}(X) )=(A,\mathscr C(X,Y),L,U) \\
@V (\mathscr C(X,Y),w)VV @VV (A,\mathscr C(X,Y),w)V \\
\matholdcal H_Y^U(X) = (\mathscr C(X,Y),U)@>\tau(\matholdcal H_Y^U)(X)>=(\zeta^\tau_{XY},U)> (A,\matholdcal H_Y^U(X))=(A,\mathscr C(X,Y),U)\\
\end{CD}
\end{equation} The result is now clear. 
 \end{proof}

 \begin{lem}\label{L5.2c}
 The association in \eqref{5.8uu} gives a well defined map $\beta: V'\longrightarrow V_1$.
 \end{lem}
 
 \begin{proof}
Since $\tau(\mathcal H_Y^U)$ is a morphism in ${^\mathscr C}Ctr$, we have the following commutative diagram
\begin{equation}\label{cd5.9ti}\small
\xymatrix{
(\mathscr C(X,Y),U)\ar[rr]^{(\mathscr C(X,Y),\epsilon_Y,U)\qquad\qquad\qquad}\ar[rrdd]_{id}&& (\mathscr C(X,Y),\matholdcal H_Y^U(Y))=(\mathscr C(X,Y),\mathscr C(Y,Y),U) \ar[dd]_{(\delta_{XYY},U)}
\ar[rrr]^{(\mathscr C(X,Y),\tau(\matholdcal H_Y^U)(Y))}&&& (\mathscr C(X,Y),A,\matholdcal H_Y^U(Y))=
(\mathscr C(X,Y),A,\mathscr C(Y,Y),U)\ar[d]^{(\psi_{XY},\mathscr C(Y,Y),U)}\\
&&  &&& (A,\mathscr C(X,Y),\mathscr C(Y,Y),U))\ar[d]^{(A,\delta_{XYY},U)}\\
&& \mathcal H_Y^U(X) =(\mathscr C(X,Y),U)\ar[rrr]_{\tau(\matholdcal H_Y^U)(X)} &&& (A,\matholdcal H_Y^U(X))=(A,\mathscr C(X,Y),U)\\
}
\end{equation} From \eqref{cd5.9ti} and the definition in \eqref{5.8uu}, we see that
\begin{equation}\label{5.10eq1}\small 
\tau(\matholdcal H_Y^U)(X)=(A,\delta_{XYY},U)\circ (\psi_{XY},\mathscr C(Y,Y),U)\circ (\mathscr C(X,Y),\tau(\matholdcal H_Y^U)(Y)) \circ (\mathscr C(X,Y),\epsilon_Y,U)=(A,\delta_{XYY},U)\circ (\psi_{XY},\mathscr C(Y,Y),U)\circ (\mathscr C(X,Y),\sigma_Y,U)
\end{equation} On the other hand, we have maps 
\begin{equation}\label{514eq}
\Delta^{XY}_U(Z):\matholdcal H_X^{(\mathscr C(X,Y),U)}(Z)=(\mathscr C(Z,X),\mathscr C(X,Y),U)\xrightarrow{\qquad(\delta_{ZXY},U)\qquad} \matholdcal H_Y^U(Z)=(\mathscr  C(Z,Y),U)
\end{equation} which induce a morphism $\Delta^{XY}_U:\matholdcal H_X^{(\mathscr C(X,Y),U)}\longrightarrow \matholdcal H_Y^U$ in  ${^\mathscr C}Ctr$. The natural transformation
$\tau$ now gives us a commutative diagram
\begin{equation}\label{515tg}\small 
\xymatrix{
&(\mathscr C(X,X),\mathscr C(X,Y),U)=  \matholdcal H_X^{(\mathscr C(X,Y),U)}(X) \ar[d]_{\Delta^U_{XY}(X)}^{=(\delta_{XXY},U)}\ar[rrr]^{\tau(\matholdcal H_X^{(\mathscr C(X,Y),U)})(X)}_{=(\mathscr C(X,Y)\otimes \zeta^\tau_{XX},U)} &&& (A,\matholdcal H_X^{(\mathscr C(X,Y),U)}(X) )=(A,\mathscr C(X,X),\mathscr C(X,Y),U)  \ar[d]_{(A,\Delta^U_{XY}(X))}^{=(A,\delta_{XXY},U)}\\
(\mathscr C(X,Y),U)\ar[r]_{id}\ar[ru]^{(\mathscr C(X,Y)\otimes \epsilon_X,U)}&   \matholdcal H_Y^U(X) =(\mathscr C(X,Y),U)\ar[rrr]_{\tau(\matholdcal H_Y^U)(X)} &&&  (A,\matholdcal H_Y^U(X))=(A,\mathscr C(X,Y),U)\\
}
\end{equation} where the equality $\tau(\matholdcal H_X^{(\mathscr C(X,Y),U)})(X)=(\mathscr C(X,Y)\otimes \zeta^\tau_{XX},U)$ follows from Lemma \ref{L5.01b}. Since $\sigma_X=\epsilon_X\circ \zeta^\tau_{XX}$, it follows from \eqref{515tg} that
\begin{equation}\label{5.10eq3} 
\tau(\matholdcal H_Y^U)(X)= (A,\delta_{XXY},U)\circ (\mathscr C(X,Y)\otimes \zeta^\tau_{XX},U)\circ (\mathscr C(X,Y)\otimes \epsilon_X,U)= (A,\delta_{XXY},U)\circ   (\sigma_X,\mathscr C(X,Y) ,U)
\end{equation} From \eqref{5.10eq1} and \eqref{5.10eq3}, it is clear that the collection $\sigma=\{\sigma_X\}_{X\in Ob(\mathscr C)}$ satisfies the condition in 
\eqref{5.3zz}, i.e., $\sigma\in V_1$. This proves the result.
 \end{proof}
 
  \begin{thm}\label{P5.5}
 The morphisms $\alpha:V_1\longrightarrow V'$ and $\beta:V'\longrightarrow V_1$ are mutually inverse isomorphisms. In particular, $V_1\cong Nat(id_{{^\mathscr C}Ctr},S_\psi T_\psi)$.
 \end{thm}

 \begin{proof}
We claim that $\beta$ is a monomorphism.  Otherwise, suppose we have $0\ne \tau\in V'$ such that $\sigma:=\beta(\tau)=0$. From \eqref{5.10eq3}, it follows immediately that $\tau(\matholdcal H_Y^U)(X)=0$ for every
 $X$, $Y\in Ob(\mathscr C)$ and $U\in Vect_K$. In particular, we have $\tau(\matholdcal H_X^*)=0$ for every $X\in Ob(\mathscr C)$. Since $\tau\ne 0$, we can choose $\matholdcal M\in {^\mathscr C}Ctr$ and $X\in Ob(\mathscr C)$ such that $\tau(\matholdcal M)(X)\ne 0$. Equivalently, we have $v:K\longrightarrow  \matholdcal M(X)$ such that $\tau(\matholdcal M)(X)
 \circ v\ne 0$. Since ${^\mathscr C}Ctr(\matholdcal H_X^*,\matholdcal M)\cong Vect_K(K,\matholdcal M(X))$ by \eqref{3.5rf}, we have a map 
 $\xi_v:\matholdcal H_X^* \longrightarrow \matholdcal M $ in ${^\mathscr C}Ctr$ such that $\xi_v(X)\circ (\epsilon_X,K)=v$. The following commutative diagram
 \begin{equation*}
 \xymatrix{
 && \matholdcal H_X^*(X)\ar[rrr]^{\tau(\matholdcal H_X^* )(X)=0} \ar[d]_{\xi_v(X)}&&& (A,\matholdcal H_X^*(X))\ar[d]^{(A,\xi_v(X))}\\
K\ar[rr]_{v} \ar[rru]^{ (\epsilon_X,K)} && \matholdcal M(X)\ar[rrr]_{\tau(\matholdcal M )(X)} &&& (A,\matholdcal M(X))\\
 }
 \end{equation*}
now shows that $\tau(\matholdcal M)(X)
 \circ v=0$, which is a contradiction.
 
 \smallskip
  From the definitions in \eqref{44r5} and \eqref{5.8uu}, we can verify that $\beta\circ \alpha=id$.  This proves the result.
 \end{proof}
 
\begin{Thm}\label{T4.5hm5}
Let $(\mathscr C,A,\psi)$ be an entwining structure. Then, the functor $T_\psi :{^{\mathscr C}}Ctr\longrightarrow  {_A^{\mathscr C}}Ctr(\psi)$ defined by setting $T_\psi (\matholdcal M):=(
A,\matholdcal M)$ is separable if and only if there exists $\sigma \in V_1$ such that
\begin{equation}\label{421condxz}
\sigma_X(f\otimes 1)=\epsilon_X(f)\qquad \forall \textrm{ }f\in \mathscr C(X,X), \textrm{ }X\in Ob(\mathscr C)
\end{equation}
\end{Thm}

\begin{proof}
Using \eqref{318g}, the counit $\omega: S_\psi T_\psi\longrightarrow id$ of the adjunction $(S_\psi,T_\psi)$ is given by 
$\omega(\matholdcal M)(X):(A,\matholdcal M(X))\xrightarrow{\qquad (u_A,\matholdcal M(X))} \matholdcal M(X)$ for $\matholdcal M\in {^{\mathscr C}}Ctr$, where
$u_A:K\longrightarrow A$ is the unit map on $A$. By \cite[Proposition 1.1]{uni}, we know that the right adjoint $T_\psi$ is separable if and only if there is
$\tau\in V'=Nat(id_{{^\mathscr C}Ctr},S_\psi T_\psi)$ such that $\omega\circ\tau=id_{{^\mathscr C}Ctr}$. From the isomorphism $V_1\cong Nat(id_{{^\mathscr C}Ctr},S_\psi T_\psi)$ described in Proposition \ref{P5.5} and the definitions in \eqref{44r5} and \eqref{5.8uu}, it may be verified that this is equivalent to the existence of $\sigma\in V_1$ such that
\begin{equation}
(K,\_\_)\xrightarrow{\qquad\qquad(\epsilon_X,\_\_)=(u_A,\mathscr C(X,X),\_\_)\circ (\sigma_X,\_\_)\qquad\qquad} (\mathscr C(X,X),\_\_)
\end{equation} for each $X\in Ob(\mathscr C)$. The result is now clear. 
\end{proof}

We will now give conditions for the functor $S_\psi:  {_A^{\mathscr C}}Ctr(\psi)\longrightarrow {^{\mathscr C}}Ctr$ to be separable. We only sketch the proof here. We begin by setting $W':=Nat(T_\psi S_\psi,id_{ {_A^{\mathscr C}}Ctr(\psi)})$. An element $\lambda\in W_1$ as defined in \eqref{422lu} can now be expressed as a collection of natural transformations
\begin{equation}\label{422lux5}
\lambda=\{\widetilde\lambda_X: (A,A,\_\_)\xrightarrow{\qquad (\lambda^X,\_\_)\qquad} (\mathscr C(X,X),\_\_) \}_{X\in Ob(\mathscr C)}
\end{equation} satisfying the following two commutative diagrams
\begin{equation}\label{cond51vn}
\xymatrix{
(A,\mathscr C(X,Y),A,\_\_) \ar[rrr]^{(A,\psi_{XY},\_\_)}&&&  (A,A,\mathscr C(X,Y),\_\_)  
\ar[rrr]^{(\lambda^X,\mathscr C(X,Y),\_\_)} &&& (\mathscr C(X,X),\mathscr C(X,Y),\_\_)\ar[d]^{(\delta_{XXY},\_\_)}\\
(\mathscr C(X,Y),A,A,\_\_) \ar[rrr]_{(\mathscr C(X,Y),\lambda^Y,\_\_)}\ar[u]^{(\psi_{XY},A,\_\_)}&&& (\mathscr C(X,Y),\mathscr C(Y,Y),\_\_) 
\ar[rrr]_{(\delta_{XYY},\_\_)}&&& (\mathscr C(X,Y),\_\_)\\
}
\end{equation} 
\begin{equation}\label{cond52vn}
\xymatrix{
(A,A,A,\_\_)\ar[rrrrrr]^{(\lambda^Z,A,\_\_)}&&& &&& (\mathscr C(Z,Z),A,\_\_)\ar[d]^{(\psi_{ZZ},\_\_)}\\
(A,A,\_\_) \ar[rrr]^{(\mu_A,A,\_\_)}\ar[u]^{(A,\mu_A,\_\_)}&&& (A,A,A,\_\_)\ar[rrr]^{(A,\lambda^Z,\_\_)}&&& (A,\mathscr C(Z,Z),\_\_)\\
}
\end{equation} for $X$, $Y$, $Z\in Ob(\mathscr C)$. 

\smallskip
We now define $\alpha:W_1\longrightarrow W'$ by setting $\kappa:=\alpha(\lambda)$ to be the natural transformation given by
\begin{equation}\label{522eqcf}
\kappa(\matholdcal M)(X):(A,\matholdcal M(X))\xrightarrow{\quad (A,\mu_X)\quad}(A,A,\matholdcal M(X))\xrightarrow{\quad(\lambda^X,\matholdcal M(X))\quad} (\mathscr C(X,X),\matholdcal M(X))
\xrightarrow{\quad\pi_{XX}\quad}\matholdcal M(X)
\end{equation}
for each $(\matholdcal M,\pi^{\matholdcal M},\mu^{\matholdcal M})\in {_A^{\mathscr C}}Ctr(\psi)$. On the other hand, consider the object
 $(A,\matholdcal H_X^U)\in {_A^{\mathscr C}}Ctr(\psi)$ for each $X\in Ob(\mathscr C)$ and $U\in Vect_K$. Then, we define $\beta:
 W'\longrightarrow W_1
 $ by setting $\lambda:=\beta(\kappa)$ as follows
 \begin{equation}\label{523eqcf}
 \widetilde\lambda^X(U):(A,A,U)\xrightarrow{\quad(A,A,\epsilon_X,U)\quad }(A,A,\mathscr C(X,X),U)\xrightarrow{\quad\kappa(A,\matholdcal H_X^U)(X)\quad }(A,\mathscr C(X,X),U)\xrightarrow{\quad(u_A,\mathscr C(X,X),U)\quad } (\mathscr C(X,X),U)
 \end{equation} It may be verified that $\alpha$ and $\beta$ are mutually inverse isomorphisms and that we have the following result.
 
\begin{Thm}\label{5T4.11hm5}
Let $(\mathscr C,A,\psi)$ be an entwining structure. Then, the functor $S_\psi : {_A^{\mathscr C}}Ctr(\psi)\longrightarrow  {^{\mathscr C}}Ctr$   is separable if and only if there exists $\lambda \in W_1$ such that
\begin{equation}\label{448condxz}
\lambda^{X1}(f)\lambda^{X2}(f)=\epsilon_X(f)\cdot 1 \in A\qquad \forall \textrm{ }f\in \mathscr C(X,X), \textrm{ }X\in Ob(\mathscr C)
\end{equation}
\end{Thm}

\section{Frobenius properties}

Given  the entwining structure $(\mathscr C,A,\psi)$, we have constructed two pairs $(S_\psi,T_\psi)$ and $(F_\psi,G_\psi)$ of adjoint functors
\begin{equation}
\begin{array}{c}
S_\psi: {_A^{\mathscr C}}Ctr(\psi)\longrightarrow {^{\mathscr C}}Ctr \qquad T_\psi:  {^{\mathscr C}}Ctr
\longrightarrow  {_A^{\mathscr C}}Ctr(\psi)\\
F_\psi: Com^{\mathscr C}\longrightarrow Com_A^{\mathscr C}(\psi) \qquad G_\psi: Com_A^{\mathscr C}(\psi)\longrightarrow Com^{\mathscr C}\\
\end{array}
\end{equation} In this section, we will give conditions for $(S_\psi,T_\psi)$ and $(F_\psi,G_\psi)$ to be Frobenius pairs, i.e., for $(T_\psi,S_\psi)$ and $(G_\psi,F_\psi)$ to also be adjoint pairs.

\smallskip
We start with the pair $(S_\psi,T_\psi)$. As in Section 5, we set $V':=Nat(id_{{^\mathscr C}Ctr},S_\psi T_\psi)$ and $W':=Nat(T_\psi S_\psi,id_{ {_A^{\mathscr C}}Ctr(\psi)})$. Using the general criterion for a pair of functors to be a Frobenius pair (see for instance, \cite[$\S$ 1]{uni}), it follows that
$(S_\psi,T_\psi)$ is a Frobenius pair if and only if there exists $\kappa\in W'=Nat(T_\psi S_\psi,id_{ {_A^{\mathscr C}}Ctr(\psi)})$ and 
$\tau\in V'=Nat(id_{{^\mathscr C}Ctr},S_\psi T_\psi)$ such that
\begin{gather}
 S_\psi(\kappa(\matholdcal M))\circ \tau_{S_\psi(\matholdcal M)}= id_{S_\psi(\matholdcal M)}
: \matholdcal M \xrightarrow{\qquad\tau_{S_\psi(\matholdcal M)}\qquad} (A,\matholdcal M)
\xrightarrow{\qquad S_\psi(\kappa(\matholdcal M))\qquad}\matholdcal M  \label{6.2cr} \\
\kappa(T_\psi(\matholdcal N))\circ T_\psi(\tau(\matholdcal N))= id_{T_\psi(\matholdcal N)}
: (A,\matholdcal N)\xrightarrow{\qquad  T_\psi(\tau(\matholdcal N)) \qquad} (A,(A,\matholdcal N))\xrightarrow{\qquad  \kappa(T_\psi(\matholdcal N))\qquad}  (A,\matholdcal N)\label{6.3cr}
\end{gather} for any $\matholdcal M\in  {_A^{\mathscr C}}Ctr(\psi)$ and $\matholdcal N\in {^{\mathscr C}}Ctr$. 

\begin{lem}\label{Lm6.1}
Suppose that $(S_\psi,T_\psi)$ is a Frobenius pair. Then, there exists $\sigma\in V_1$ and $\lambda\in W_1$ such that 
\begin{gather}
\epsilon_X(f)\cdot 1= \sigma_X(f_{X1}\otimes \lambda^{X1}(f_{X2}))\lambda^{X2}(f_{X2})\label{6.4cr}\\
\epsilon_X(f)\cdot 1= \sigma_X(f_{X1}^\psi\otimes \lambda^{X2}(f_{X2}))\lambda^{X1}(f_{X2})_\psi\label{6.5cr} 
\end{gather} for every $f\in \mathscr C(X,X)$, $X\in Ob(\mathscr C)$. 
\end{lem}

\begin{proof}
Since  $(S_\psi,T_\psi)$ is a Frobenius pair, we know there exist $\kappa\in W'$ and 
$\tau\in V' $ satisfying \eqref{6.2cr} and \eqref{6.3cr}.   Applying the results in Section 5, we choose $\sigma\in V_1\cong V'$ and $\lambda\in W_1\cong W'$ corresponding respectively to $\tau\in V'$ and $\kappa\in W'$. Applying \eqref{6.2cr} to $(A,\matholdcal H_X^U)\in {_A^{\mathscr C}}Ctr(\psi)$ and using the definitions in \eqref{44r5} and \eqref{522eqcf}, we see that the following diagram commutes
\begin{equation*}
\small
\xymatrix{
(A,\mathscr C(X,X),U)\ar[rr]^{(\sigma_X,A,\mathscr C(X,X),U)}\ar[dd]_{id}&& (A,\mathscr C(X,X),A,\mathscr C(X,X),U)
\ar[rr]^{(A,\psi_{XX},\mathscr C(X,X),U)}&& (A,A,\mathscr C(X,X),\mathscr C(X,X),U) \ar[d]^{ (A,A,\delta_{XXX},U)}\\
&& && (A,A,\mathscr C(X,X),U)\ar[d]^{(A,\mu_A,\mathscr C(X,X),U)}\\
(A,\mathscr C(X,X),U)&& \ar[ll]_{(A,\delta_{XXX},U)\qquad\qquad\qquad\qquad}(A,\mathscr C(X,X),\mathscr C(X,X),U)\xleftarrow{(\psi_{XX},\mathscr C(X,X),U)} (\mathscr C(X,X),A,\mathscr C(X,X),U)&&\ar[ll]_{\qquad\qquad\qquad\qquad(\lambda^X,A,\mathscr C(X,X),U)} (A,A,A,\mathscr C(X,X),U)\\
}
\end{equation*}
for every $X\in Ob(\mathscr C)$ and $U\in Vect_K$.   Using Yoneda lemma, it follows from the above diagram that we have
\begin{equation}\label{6.6eqt}
f\otimes 1= \sigma_X(f_{X1X2}^\psi\otimes \lambda^{X2}(f_{X2}))f_{X1X1}\otimes \lambda^{X1}(f_{X2})_\psi
\end{equation} for each $f\in \mathscr C(X,X)$. Applying $\epsilon_X$ to both sides of \eqref{6.6eqt}, we obtain the relation in \eqref{6.5cr}. Further, if we apply \eqref{6.3cr} to 
the object $\matholdcal H_X^U\in {^\mathscr C}Ctr$ and use the definitions in \eqref{44r5} and \eqref{522eqcf}, we see that the following diagram commutes
\begin{equation*}
\small
\xymatrix{
(A,\mathscr C(X,X),U)\ar[rr]^{(A,\sigma_X,\mathscr C(X,X),U)}\ar[d]_{id}&& (A,A,\mathscr C(X,X),\mathscr C(X,X),U)
\ar[rr]^{(A,A,\delta_{XXX},U)}&& (A,A,\mathscr C(X,X),U) \ar[d]_{ (A,\mu_A,\mathscr C(X,X),U)}\\
(A,\mathscr C(X,X),U)&& \ar[ll]_{(A,\delta_{XXX},U)\qquad\qquad\qquad\qquad}(A,\mathscr C(X,X),\mathscr C(X,X),U)\xleftarrow{(\psi_{XX},\mathscr C(X,X),U)} (\mathscr C(X,X),A,\mathscr C(X,X),U)&&\ar[ll]_{\qquad\qquad\qquad\qquad(\lambda^X,A,\mathscr C(X,X),U)} (A,A,A,\mathscr C(X,X),U)\\
}
\end{equation*}
for every $X\in Ob(\mathscr C)$ and $U\in Vect_K$.  Again using Yoneda lemma, it follows from the above diagram that we have
\begin{equation}\label{6.7eqt}
f\otimes 1=\sigma_X(f_{X1X2}\otimes \lambda^{X1}(f_{X2}))f_{X1X1}\otimes \lambda^{X2}(f_{X2})
\end{equation} for each $f\in \mathscr C(X,X)$.  Applying $\epsilon_X$ to both sides of \eqref{6.7eqt}, we obtain the relation in \eqref{6.4cr}. 
\end{proof}

\begin{Thm}\label{P6.2}
Let $(\mathscr C,A,\psi)$ be an entwining structure. Then, the functors
\begin{equation}
S_\psi: {_A^{\mathscr C}}Ctr(\psi)\longrightarrow {^{\mathscr C}}Ctr \qquad T_\psi:  {^{\mathscr C}}Ctr
\longrightarrow  {_A^{\mathscr C}}Ctr(\psi)
\end{equation} form a Frobenius pair if and only if there exists $\sigma\in V_1$ and $\lambda\in W_1$ such that 
\begin{gather}
\epsilon_X(f)\cdot 1= \sigma_X(f_{X1}\otimes \lambda^{X1}(f_{X2}))\lambda^{X2}(f_{X2})\label{6.4crt}\\
\epsilon_X(f)\cdot 1= \sigma_X(f_{X1}^\psi\otimes \lambda^{X2}(f_{X2}))\lambda^{X1}(f_{X2})_\psi\label{6.5crt} 
\end{gather} for every $f\in \mathscr C(X,X)$, $X\in Ob(\mathscr C)$. 
\end{Thm}

\begin{proof}  The only if part follows from Lemma \ref{Lm6.1}. We suppose therefore that  there exists $\sigma\in V_1$ and $\lambda\in W_1$ satisfying 
\eqref{6.4crt} and \eqref{6.5crt}. We have $\tau\in V'\cong V_1$ corresponding to $\sigma\in V_1$ and $\kappa\in W'\cong W_1$ corresponding to $\lambda\in W_1$. For $
\matholdcal N\in {^{\mathscr C}}Ctr$, we first check the condition in \eqref{6.3cr}. For $X\in Ob(\mathscr C)$, it follows from \eqref{44r5} that $ T_\psi(\tau(\matholdcal N))(X)$ may be expressed as the composition 
\begin{equation}\label{611ctcd}
\begin{CD}
(A,\matholdcal N(X))@>(A,\sigma_X,\matholdcal N(X))>>(A,A,\mathscr C(X,X),\matholdcal N(X)) @>(A,A,\pi_{XX})>> (A,A,\matholdcal N(X))
\end{CD}
\end{equation} Similarly, it follows from \eqref{522eqcf} that $ \kappa(T_\psi(\matholdcal N))(X)$  may be expressed as the composition 
\begin{equation}\label{612ctcd}{\small
\begin{CD}
 (A,A,\matholdcal N(X))@>(A,\mu_A,\matholdcal N(X))>> (A,A,A,\matholdcal N(X))@>(\lambda^X,A,\matholdcal N(X))>> (\mathscr C(X,X),A,\matholdcal N(X))
 @>(\psi_{XX},\matholdcal N(X))>> (A,\mathscr C(X,X),\matholdcal N(X)) @>(A,\pi_{XX})>> (A,\matholdcal N(X))
\end{CD}}
\end{equation} We now see that the composition of \eqref{611ctcd} and \eqref{612ctcd} may be expressed as the composition 
\begin{equation}\label{6.13eqt} (A,\matholdcal N(X))\xrightarrow{(\upsilon_X,\matholdcal N(X))}(A,\mathscr C(X,X),\matholdcal N(X))\xrightarrow{(A,\pi_{XX})}(A,\matholdcal N(X))
\end{equation} where
\begin{equation}\label{ct613cd}
\upsilon_X:\mathscr C(X,X)\otimes A\longrightarrow A\qquad (f\otimes a)\mapsto \sigma_X(f_{X1}\otimes a_\psi\lambda^{X1}(f_{X2}^\psi))\lambda^{X2}(f_{X2}^\psi)
\end{equation} Since both \eqref{611ctcd} and \eqref{612ctcd} are $A$-linear, it follows that the following diagram commutes
\begin{equation}\label{615cd}\small 
\xymatrix{
(A,\matholdcal N(X))\ar[rrrr]^{(\upsilon_X,\matholdcal N(X))}\ar[dd]_{(\mu_A,\matholdcal N(X))}
\ar[rrrrdd]^{(A,\epsilon_X,\matholdcal N(X))}&&&& (A,\mathscr C(X,X),\matholdcal N(X)) \ar[rr]^{(A,\pi_{XX})} && (A,\matholdcal N(X))
\ar[d]^{(\mu_A,\matholdcal N(X))}\\
&&&&&& (A,A,\matholdcal N(X))\ar[d]^{(A,u_A,\matholdcal N(X))}\\
(A,A,\matholdcal N(X)) \ar[rr]_{(A,\upsilon_X,\matholdcal N(X))}&& (A,A,\mathscr C(X,X),\matholdcal N(X)) \ar[rr]_{ (A,u_A,\mathscr C(X,X),\matholdcal N(X))} &&
(A,\mathscr C(X,X),\matholdcal N(X)) \ar[rr]_{(A,\pi_{XX})}& & (A,\matholdcal N(X))\\
}
\end{equation} Here, the lower triangle in \eqref{615cd} commutes because it follows from \eqref{6.4crt} that
\begin{equation}\label{rel616}
\upsilon_X(f\otimes 1)a= \sigma_X(f_{X1}\otimes \lambda^{X1}(f_{X2}))\lambda^{X2}(f_{X2})a=\epsilon_X(f)a
\end{equation} for $f\in \mathscr C(X,X)$, $a\in A$. Since $(A,\pi_{XX})\circ (A,\epsilon_X,\matholdcal N(X))=id$ and $(A,u_A,\matholdcal N(X))\circ (\mu_A,\matholdcal N(X))=id$, it is now clear from \eqref{615cd} that the composition in \eqref{6.13eqt} is the identity. 

\smallskip
Finally, we consider $(\matholdcal M,\mu^{\matholdcal M},\pi^{\matholdcal M})\in {_A^{\mathscr C}}Ctr(\psi)$. Then, we have the following commutative diagram
\begin{equation}\label{617com}\small
\xymatrix{
\matholdcal M(X) \ar[rrrrr]^{(\sigma_X,\matholdcal M(X))}\ar[rrrrrdddd]_{(\epsilon_X,\matholdcal M(X))}&&&&& (A,\mathscr C(X,X),\matholdcal M(X))\ar[d]^{(A,\mathscr C(X,X),\mu_X)} \ar[rrr]^{(A,\pi_{XX})}&&& 
(A,\matholdcal M(X))\ar[dd]^{(A,\mu_X)}
\\  &&&&& (A,\mathscr C(X,X),A,\matholdcal M(X))
\ar[d]^{(A,\psi_{XX},\matholdcal M(X))}&&&\\ &&&&&  (A,A,\mathscr C(X,X),\matholdcal M(X))\ar[rrr]_{(A,A,\pi_{XX})}\ar[d]^{(\lambda^X,\mathscr C(X,X),\matholdcal M(X))}&&& (A,A,\matholdcal M(X)) 
\ar[d]^{(\lambda^X,\matholdcal M(X))}\\
&&&&& (\mathscr C(X,X),\mathscr C(X,X),\matholdcal M(X)) \ar[d]^{(\delta_{XXX},\matholdcal M(X))}\ar[rrr]_{(\mathscr C(X,X),\pi_{XX})}&&& (\mathscr C(X,X),\matholdcal M(X)) 
\ar[d]^{\pi_{XX}}\\ &&&&& (\mathscr C(X,X),\matholdcal M(X))\ar[rrr]_{\pi_{XX}}&&&\matholdcal M(X)\\
}
\end{equation} Here, the triangle in \eqref{617com} commutes due to the relation in \eqref{6.5crt}. From the definitions in \eqref{44r5} and \eqref{522eqcf}, it is now clear that the composition in \eqref{6.2cr} applied to $(\matholdcal M,\mu^{\matholdcal M},\pi^{\matholdcal M})\in {_A^{\mathscr C}}Ctr(\psi)$ gives the identity.
\end{proof}

We now come to entwined comodules and the adjoint pair $(F_\psi,G_\psi)$. Again using the general criterion for a pair of functors to be Frobenius (see, for instance, 
\cite[$\S$ 1]{uni}), we see that $(F_\psi,G_\psi)$ is a Frobenius pair if and only if there exist $\tau\in V=Nat(G_\psi F_\psi ,id_{Com^{\mathscr C}})$ and 
$\kappa\in W=Nat(id_{Com_A^{\mathscr C}(\psi)},F_\psi G_\psi )$ such that 
\begin{gather}
 F_\psi(\tau(\mathcal N))\circ \kappa_{F_\psi(\mathcal N)}= id_{F_\psi(\mathcal N)}
: \mathcal N\otimes A \xrightarrow{\qquad \kappa_{F_\psi(\mathcal N)}\qquad} \mathcal N\otimes A\otimes A
\xrightarrow{\qquad F_\psi(\tau(\mathcal N))\qquad}\mathcal N\otimes A  \label{6.2cr7} \\
\tau(G_\psi(\mathcal M))\circ G_\psi(\kappa(\mathcal M))= id_{G_\psi(\mathcal M)}
: \mathcal M\xrightarrow{\qquad G_\psi(\kappa(\mathcal M))  \qquad} \mathcal M\otimes A\xrightarrow{\qquad \tau(G_\psi(\mathcal M))\qquad}  \mathcal M\label{6.3cr7}
\end{gather}
for $\mathcal M\in Com_A^{\mathscr C}(\psi)$ and $\mathcal N\in Com^{\mathscr C}$. We can now prove the following result.

\begin{Thm}\label{T6.27}
Let $(\mathscr C,A,\psi)$ be an entwining structure. Then, the functors
\begin{equation}
F_\psi: Com^{\mathscr C}\longrightarrow Com_A^{\mathscr C}(\psi) \qquad G_\psi: Com_A^{\mathscr C}(\psi)\longrightarrow Com^{\mathscr C}
\end{equation} form a Frobenius pair if and only if there exists $\sigma\in V_1$ and $\lambda\in W_1$ such that 
\begin{gather}
\epsilon_X(f)\cdot 1= \sigma_X(f_{X1}\otimes \lambda^{X1}(f_{X2}))\lambda^{X2}(f_{X2})\label{6.4crt7}\\
\epsilon_X(f)\cdot 1= \sigma_X(f_{X1}^\psi\otimes \lambda^{X2}(f_{X2}))\lambda^{X1}(f_{X2})_\psi\label{6.5crt7} 
\end{gather} for every $f\in \mathscr C(X,X)$, $X\in Ob(\mathscr C)$. 
\end{Thm}

\begin{proof}
The only if part follows by applying \eqref{6.2cr7} and \eqref{6.3cr7} to $\mathcal H_X\in Com^{\mathscr C}$  and 
$\mathcal H_X\otimes A\in  Com_A^{\mathscr C}(\psi)$ for $X\in Ob(\mathscr C)$. The if part may also be verified directly.
\end{proof}

\section{Maschke type theorems}

In this final section, our objective is to prove Maschke type results for the categories $ {_A^{\mathscr C}}Ctr(\psi)$ and $Com_A^{\mathscr C}(\psi)$. In the case of modules over an ordinary entwining structure, such results were proved by Brzezi\'nski \cite{Brz1} and earlier by Caenepeel, Militiaru and Zhu \cite{CMZ} for the particular case of Doi-Hopf modules. 

\smallskip
Throughout this section, we will assume that the algebra $A$ is finite dimensional as a vector space over $K$. In particular, this means that we have a standard coevaluation map
\begin{equation}\label{coev7}
coev_A: K\longrightarrow A\otimes A^*\qquad t\mapsto t\sum_{i\in I} a_i\otimes a_i^*
\end{equation} where $\{a_i\}_{i\in I}$ is a basis of $A$, and $\{a^*_i\}_{i\in I}$ its dual basis. It is well known that the coevaluation map is independent of the choice of the basis
$\{a_i\}_{i\in I}$. The following  now extends \cite[Definition 4.5]{Brz1} to the case of coalgebras with several objects.

\begin{defn}\label{D7.1}
Let $(\mathscr C,A,\psi)$ be an entwining structure and let $A$ be finite dimensional over $K$. A normalized cointegral on $(\mathscr C,A,\psi)$ consists of a collection of
$K$-linear maps
\begin{equation}\label{norm7}
\gamma=\{\gamma_{X}:A^*\otimes\mathscr C(X,X)\longrightarrow A\}_{X\in Ob(\mathscr C)}
\end{equation} satisfying the following three conditions (for $X$, $Y\in Ob(\mathscr C)$:

\smallskip
(1) $(A\otimes \psi_{XY})\circ (\psi_{XY}\otimes\gamma_X)\circ (\mathscr C(X,Y)\otimes coev_A\otimes \mathscr C(X,X))\circ \delta_{XXY}=(A\otimes \gamma_Y\otimes \mathscr C(X,Y))\circ (coev_A\otimes \delta_{XYY})$ which can also be expressed as
\begin{equation}\label{7.2e}
\begin{CD}
(\mathscr C(X,Y),A,A,\_\_) @>\qquad (\delta_{XXY},\_\_)\circ (\mathscr C(X,X),coev_A,\mathscr C(X,Y),\_\_)\circ (\gamma_X,\psi_{XY},\_\_)\circ (\psi_{XY},A,\_\_)\qquad >=(\delta_{XYY},coev_A,\_\_)\circ  (\mathscr C(X,Y),\gamma_Y,A,\_\_)>   (\mathscr C(X,Y),\_\_) 
\end{CD}
\end{equation}

\smallskip
(2) $ (A\otimes \mu_A)\circ(A\otimes \gamma_X\otimes A)\circ (coev_A\otimes \mathscr C(X,X)\otimes A)=(\mu_A\otimes \gamma_X)\circ (A\otimes coev_A
\otimes \mathscr C(X,X))\circ \psi_{XX}$  which can also be expressed as
\begin{equation}\label{7.3e}
\begin{CD}
(A,A,\_\_)  @>\qquad  (A,
\mathscr C(X,X),coev_A,\_\_)\circ (A,\gamma_X,A,\_\_)\circ (\mu_A,A,\_\_)\qquad>= (\psi_{XX},\_\_)\circ (\mathscr C(X,X),coev_A,A,\_\_)\circ (\gamma_X,\mu_A,\_\_)>  (A,\mathscr C(X,X),\_\_) 
\end{CD}
\end{equation}

\smallskip
(3) $\mu_A\circ (A\otimes \gamma_X)\circ (coev_A\otimes \mathscr C(X,X))=u_A\circ \epsilon_X$ which can also be expressed as
\begin{equation}\label{7.4e}
\begin{CD}
(A,\_\_)@>\qquad   (\mathscr C(X,X),coev_A,\_\_)\circ (\gamma_X,A,\_\_)\circ (\mu_A,\_\_)\qquad>=  (\epsilon_X,\_\_)\circ (u_A,\_\_)> (\mathscr C(X,X),\_\_)
\end{CD} 
\end{equation} 
\end{defn}

We now have the following result.

\begin{lem}\label{L7.2}
Let $(\matholdcal M,\pi,\mu)$,  $(\matholdcal M',\pi',\mu')$ be objects in 
${_A^{\mathscr C}}Ctr(\psi)$ and let $\phi:\matholdcal M\longrightarrow\matholdcal M'$ be a morphism in ${^{\mathscr C}}Ctr$. Then, the collection
$\{\widehat\phi(Y)\}_{Y\in Ob(\mathscr C)}$  of morphisms defined by the compositions
\begin{equation}\label{comp7.5}
\xymatrix{
\matholdcal M(Y) \ar[d]_{\widehat\phi(Y)}\ar[rr]^{\mu_Y} &&  (A,\matholdcal M(Y)) \ar[rr]^{(A,\phi(Y))} && (A,\matholdcal M'(Y))\ar[rr]^{(A,\mu_Y')}&& (A,A,
\matholdcal M'(Y)) \ar[d]^{(\gamma_Y,A,\matholdcal M'(Y))}\\
\matholdcal M'(Y) &&\ar[ll]^{\pi'_{YY}}(\mathscr C(Y,Y),\matholdcal M'(Y))  && && \ar[llll]^{(\mathscr C(Y,Y),coev_A,\matholdcal M'(Y))}(\mathscr C(Y,Y),A^*,A,\matholdcal M'(Y))\\
}
\end{equation}
induces a map $\widehat\phi:\matholdcal M\longrightarrow \matholdcal M'$ in $ {_A^{\mathscr C}}Ctr(\psi)$.
\end{lem}

\begin{proof}
We will first show that $\widehat{\phi}$ is well behaved with respect to the contramodule structures.  Since $\matholdcal M\in {_A^{\mathscr C}}Ctr(\psi)$ is an entwined contramodule, we know that the maps $\mu_Y:\matholdcal M(Y)\longrightarrow (A,\matholdcal M(Y))$ are well behaved with respect 
to the contramodule structures on $\matholdcal M$ and $(A,\matholdcal M)$.  We also know that $\phi$ is already a morphism in  ${^{\mathscr C}}Ctr$. It therefore suffices to show that the following diagram commutes
\begin{equation*}\small
\xymatrix{
(\mathscr C(X,Y),A,A,\matholdcal M'(Y)) \ar[d]^{(\mathscr C(X,Y),\gamma_Y,A,\matholdcal M'(Y))}\ar[rr]^{(\psi_{XY},A,\matholdcal M'(Y))}&&  (A,\mathscr C(X,Y),A,\matholdcal M'(Y)) \ar[rr]^{(A,\psi_{XY},\matholdcal M'(Y))}& & (A,A,\mathscr C(X,Y),\matholdcal M'(Y))
\ar[d]_{(\gamma_X,A,\mathscr C(X,Y),\matholdcal M'(Y))}
\ar[rr]^{(A,A,\pi'_{XY})} &&   (A,A,\matholdcal M'(X)) \ar[d]_{(\gamma_X,A,\matholdcal M'(X))}
\\ (\mathscr C(X,Y),\mathscr C(Y,Y),A^*,A,\matholdcal M'(Y)) \ar[d]^{(\mathscr C(X,Y),\mathscr C(Y,Y),coev_A,\matholdcal M'(Y))} &&&   &
\bullet 
\ar[d]_{(\mathscr C(X,X),coev_A,\mathscr C(X,Y),\matholdcal M'(Y))}\ar[rr]^{(\mathscr C(X,X),A^*,A,\pi'_{XY})\qquad}&& (\mathscr C(X,X),A^*,A,\matholdcal M'(X)) \ar[d]_{(\mathscr C(X,X),coev_A,\matholdcal M'(X))} \\
(\mathscr C(X,Y),\mathscr C(Y,Y),\matholdcal M'(Y))  \ar[d]^{(\mathscr C(X,Y),\pi'_{YY})}&&&  &
\bullet \ar[d]^{(\delta_{XXY},\matholdcal M'(Y))} \ar[rr]^{(\mathscr C(X,X),\pi'_{XY})\quad}&&   (\mathscr C(X,X),\matholdcal M'(X)) 
\ar[d]_{\pi'_{XX}}\\
(\mathscr C(X,Y),\matholdcal M'(Y))\ar[rrrr]_{id}&&&  & (\mathscr C(X,Y),\matholdcal M'(Y))\ar[rr]_{\pi'_{XY}} &&   \matholdcal M'(X)  \\
}
\end{equation*} In the above diagram, we have suppressed some of the objects and replaced them by $\bullet$ where these are clear from the source and target
of the morphisms involved. We now check that
\begin{equation*}
\begin{array}{ll}
\pi'_{XY}\circ (\mathscr C(X,Y),\pi'_{YY})\circ (\mathscr C(X,Y),\mathscr C(Y,Y),coev_A,\matholdcal M'(Y))\circ (\mathscr C(X,Y),\gamma_Y,A,\matholdcal M'(Y)) &\\
=\pi'_{XY}\circ (\delta_{XYY},\matholdcal M'(Y))\circ (\mathscr C(X,Y),\mathscr C(Y,Y),coev_A,\matholdcal M'(Y))\circ (\mathscr C(X,Y),\gamma_Y,A,\matholdcal M'(Y)) &\\
=\pi'_{XY}\circ (\delta_{XXY},\matholdcal M'(Y))\circ (\mathscr C(X,X),coev_A,\mathscr C(X,Y),\matholdcal M'(Y))\circ (\gamma_X,
\psi_{XY},\matholdcal M'(Y)) \circ (\psi_{XY},A,
\matholdcal M'(Y))& \mbox{(by \eqref{7.2e})}\\
=\pi'_{XX}\circ (\mathscr C(X,X),coev_A,\matholdcal M'(X)) \circ (\gamma_X,A,\matholdcal M'(X))\circ (A,A,\pi'
_{XY})\circ  (A,\psi_{XY},\matholdcal M'(Y)) \circ (\psi_{XY},A,\matholdcal M'(Y))&\\
\end{array}
\end{equation*}
Similarly, in order to show that each $\widehat\phi(Y)$ is $A$-linear it suffices to check that the following diagram commutes. 
\begin{equation*}\small 
\xymatrix{
(A,\matholdcal M'(Y)) 
\ar[dd]^{(\mu_A,\matholdcal M'(Y))}\ar[rr]^{(A,\mu'_Y)}&&(A,A,\matholdcal M'(Y)) \ar[d]^{(A,A,\mu'_Y)}\ar[rrr]^{(\gamma_Y,A,\matholdcal M'(Y))} 
&&&  (\mathscr C(Y,Y),A^*,A,\matholdcal M'(Y))\ar[rrr]^{(\mathscr C(Y,Y),coev_A,\matholdcal M'(Y))} 
\ar[d]^{(\mathscr C(Y,Y),A^*,A,\mu'_Y)} &&& (\mathscr C(Y,Y),\matholdcal M'(Y)) \ar[d]^{(\mathscr C(Y,Y),\mu'_Y)}\\
  &&(A,A,A,\matholdcal M'(Y)) \ar[rrr]^{(\gamma_Y,A,A,\matholdcal M'(Y))} &&& (\mathscr C(Y,Y),A^*,A,A,\matholdcal M'(Y)) \ar[rrr]^{(\mathscr C(Y,Y),coev_A,A,\matholdcal M'(Y))}  &&& (\mathscr C(Y,Y),A,\matholdcal M'(Y))
\ar[d]^{(\psi_{YY},\matholdcal M'(Y))}\\
(A,A,\matholdcal M'(Y)) \ar[rr]^{(A,A,\mu'_Y)}&& (A,A,A,\matholdcal M'(Y)) \ar[rrr]^{(A,\gamma_Y,A,\matholdcal M'(Y))} &&&  (A,\mathscr C(Y,Y),A^*,A,\matholdcal M'(Y))\ar[rrr]^{(A,\mathscr C(Y,Y),coev_A,\matholdcal M'(Y))}  &&& (A,\mathscr C(Y,Y),\matholdcal M'(Y)) \\
}
\end{equation*} We now verify that
\begin{equation*}
\begin{array}{ll}
(A,\mathscr C(Y,Y),coev_A,\matholdcal M'(Y))\circ (A,\gamma_Y,A,\matholdcal M'(Y)) \circ  (A,A,\mu'_Y)\circ (\mu_A,
\matholdcal M'(Y))& \\
=(A,\mathscr C(Y,Y),coev_A,\matholdcal M'(Y))\circ (A,\gamma_Y,A,\matholdcal M'(Y)) \circ  (\mu_A,A,\matholdcal M'(Y))\circ (A,
\mu'_Y)& \\
=(\psi_{YY},\matholdcal M'(Y))\circ (\mathscr C(Y,Y),coev_A,A,\matholdcal M'(Y))\circ (\gamma_Y,\mu_A,
\matholdcal M'(Y))\circ (A,
\mu'_Y)&   \mbox{(by \eqref{7.3e})} \\
=(\psi_{YY},\matholdcal M'(Y))\circ (\mathscr C(Y,Y),coev_A,A,\matholdcal M'(Y))\circ (\gamma_Y,A,A,
\matholdcal M'(Y))\circ (A,A,\mu'_Y)\circ (A,
\mu'_Y)&\\
= (\psi_{YY},\matholdcal M'(Y))\circ (\mathscr C(Y,Y),\mu'_Y) \circ (\mathscr C(Y,Y),coev_A,\matholdcal M'(Y)) \circ 
(\gamma_Y,A,\matholdcal M'(Y)) \circ  (A,
\mu'_Y) &\\
\end{array}
\end{equation*} This proves the result. 
\end{proof}

\begin{lem}\label{L7.3} Let  $(\matholdcal M,\pi,\mu)$,  $(\matholdcal M',\pi',\mu')\in {_A^{\mathscr C}}Ctr(\psi)$ and let $\phi:\matholdcal M\longrightarrow\matholdcal M'$ be a morphism in ${^{\mathscr C}}Ctr$.

\smallskip
(a)  If $\phi:\matholdcal M\longrightarrow\matholdcal M'$ is already a morphism in ${_A^{\mathscr C}}Ctr(\psi)$, then 
$\widehat\phi=\phi$.

\smallskip
(b) If $\phi_P:\matholdcal P\longrightarrow \matholdcal M$ and $\phi_Q:\matholdcal M'\longrightarrow \matholdcal Q$ are morphisms in 
$ {_A^{\mathscr C}}Ctr(\psi)$, then $\widehat{\phi\circ\phi_P}=\widehat\phi\circ \phi_P$ and $\widehat{\phi_Q\circ \phi}=
\phi_Q\circ \widehat\phi$.
\end{lem}
\begin{proof}
We first prove (a). If each $\phi(Y)$ in \eqref{comp7.5} is also $A$-linear, we observe that $(A,\phi(Y))\circ \mu_Y=\mu'_Y\circ \phi(Y)$. Accordingly, it follows from \eqref{comp7.5} that
\begin{equation*}
\begin{array}{lll}
\widehat\phi(Y)&=\pi'_{YY}\circ (\mathscr C(Y,Y),coev_A,\matholdcal M'(Y))\circ  (\gamma_Y,A,\matholdcal M'(Y))\circ (A,\mu'_Y)\circ 
(A,\phi(Y))\circ \mu_Y &\\
&=\pi'_{YY}\circ (\mathscr C(Y,Y),coev_A,\matholdcal M'(Y))\circ  (\gamma_Y,A,\matholdcal M'(Y))\circ (A,\mu'_Y)\circ 
\mu'_Y\circ \phi(Y) &\\
&=\pi'_{YY}\circ (\mathscr C(Y,Y),coev_A,\matholdcal M'(Y))\circ  (\gamma_Y,A,\matholdcal M'(Y))\circ (\mu_A,\matholdcal M'(Y))\circ 
\mu'_Y\circ \phi(Y) &\\
&=\pi'_{YY}\circ (\epsilon_Y,\matholdcal M'(Y))\circ (u_A,\matholdcal M'(Y))\circ 
\mu'_Y\circ \phi(Y) &  \mbox{(by \eqref{7.4e})}\\
&=\phi(Y)&\\
\end{array}
\end{equation*}The result of (b) may also be verified directly. \end{proof}

\begin{thm}\label{P7.4}
Let $(\mathscr C,A,\psi)$ be an entwining structure such that there exists a normalized cointegral  $\gamma=\{\gamma_{X}:A^*\otimes\mathscr C(X,X)\longrightarrow A\}_{X\in Ob(\mathscr C)}$. Then, a morphism $\phi:\matholdcal M\longrightarrow\matholdcal M'$  in ${_A^{\mathscr C}}Ctr(\psi)$ has a section (resp. a retraction) in ${_A^{\mathscr C}}Ctr(\psi)$ if and only if it has a section (resp. a retraction) in ${^{\mathscr C}}Ctr$.
\end{thm}

\begin{proof}
Suppose that $\phi$ has a section $\xi:\matholdcal M'\longrightarrow\matholdcal M$  in ${^{\mathscr C}}Ctr$, i.e., $\phi\circ\xi=id$. By Lemma 
\ref{L7.2}, we see that  $\widehat\xi:\matholdcal M'\longrightarrow\matholdcal M$ is a morphism in ${_A^{\mathscr C}}Ctr(\psi)$. By Lemma \ref{L7.3},
$\phi\circ \widehat\xi=\widehat{\phi\circ\xi}=id$ and hence $\phi$ has a section in ${_A^{\mathscr C}}Ctr(\psi)$. The result for the retraction follows similarly. 
\end{proof}

We now recall the following two notions from \cite{CM2003}.

\begin{defn}\label{D7.4} (see \cite[$\S$ 3]{CM2003})
Let $F:\mathscr D\longrightarrow \mathscr D'$ be a functor between abelian categories. 

\smallskip
(a) $F$ is said to be semisimple if it satisfies the following condition: any short exact sequence $0\longrightarrow X\longrightarrow Y\longrightarrow Z\longrightarrow 0$ in 
$\mathscr D$ such that $0\longrightarrow F(X)\longrightarrow F(Y)\longrightarrow F(Z)\longrightarrow 0$ is split exact in $\mathscr D'$ must split in $\mathscr D$. 

\smallskip
(b) $F$ is said to be Maschke if it satisfies the following condition: for any morphisms $i:X\longrightarrow X'$, $f:X\longrightarrow Y$  in $\mathscr D$ such that $F(i):F(X)\longrightarrow F(X')$ is a split monomorphism in 
$\mathscr D'$, there exists $g:X'\longrightarrow Y$ such that $g\circ i=f$. 
\end{defn}

We are now ready to prove the main result of this section for entwined contramodules. 

\begin{Thm}\label{T7.5}
Let $(\mathscr C,A,\psi)$ be an entwining structure. Suppose there exists a normalized cointegral  $\gamma=\{\gamma_{X}:A^*\otimes\mathscr C(X,X)\longrightarrow A\}_{X\in Ob(\mathscr C)}$ on $(\mathscr C,A,\psi)$. Then, we have

\smallskip
(a) The  functor $S_\psi: {_A^{\mathscr C}}Ctr(\psi)\longrightarrow {^{\mathscr C}}Ctr$ is a semisimple functor.

\smallskip
(b) The  functor $S_\psi: {_A^{\mathscr C}}Ctr(\psi)\longrightarrow {^{\mathscr C}}Ctr$ is a Maschke functor.

\end{Thm}

\begin{proof}
Let $0\longrightarrow \matholdcal M'\longrightarrow \matholdcal M\longrightarrow \matholdcal M''\longrightarrow 0$ be a short exact sequence in 
$ {_A^{\mathscr C}}Ctr(\psi)$ that splits in $ { ^{\mathscr C}}Ctr$. Applying Proposition \ref{P7.4}, it is clear that the short exact sequence also 
splits in $ {_A^{\mathscr C}}Ctr(\psi)$. This proves (a). Also it is clear that $S_\psi$ reflects monomorphisms, i.e., a morphism $\phi$ in $ {_A^{\mathscr C}}Ctr(\psi)$ is a monomorphism if  $S_\psi(\phi)$ is a monomorphism. The result of (b) now follows by applying \cite[Proposition 3.7]{CM2003}. 
\end{proof}

We can also establish similar results for the category  $Com_A^{\mathscr C}(\psi)$ of entwined comodules over the datum  $(\mathscr C,A,\psi)$. More precisely, let 
 $(\mathcal M,\rho,\mu)$,  $(\mathcal M',\rho',\mu')$ be objects in 
$Com_A^{\mathscr C}(\psi)$ and let $\phi:\mathcal M\longrightarrow\mathcal M'$ be a morphism in $Com^{\mathscr C}$. Then, we can verify that the collection
$\{\check\phi(Y)\}_{Y\in Ob(\mathscr C)}$  of morphisms defined by the compositions (compare \cite[Lemma 4.7]{Brz1})
\begin{equation}\label{comp7.6}
\xymatrix{
\mathcal M(Y) \ar[d]_{\check\phi(Y)}\ar[rr]^{\rho_{YY}} &&  \mathcal M(Y)\otimes \mathscr C(Y,Y) \ar[rrrr]^{ \mathcal M(Y)\otimes coev_A\otimes \mathscr C(Y,Y) } &&  &&  \mathcal M(Y)\otimes A\otimes A^*\otimes  \mathscr C(Y,Y)  \ar[d]^{\mathcal M(Y)\otimes A\otimes \gamma_Y}\\
\mathcal M'(Y) &&\ar[ll]^{\mu'_{Y}} \mathcal M'(Y)\otimes A  && \ar[ll]^{\phi(Y)\otimes A}  \mathcal M(Y)\otimes A && \ar[ll]^{\mu_Y\otimes A} \mathcal M(Y)\otimes A\otimes A\\
}
\end{equation}
induces a map $\check\phi:\mathcal M\longrightarrow \mathcal M'$ in $ Com_A^{\mathscr C}(\psi)$.  We can also check that the association $\phi\mapsto \check\phi$ is well behaved with respect to compositions in a manner similar to Lemma \ref{L7.3}. As in Proposition \ref{P7.4}, this gives the following result.

\begin{thm}\label{P7.6}
Let $(\mathscr C,A,\psi)$ be an entwining structure such that there exists a normalized cointegral  $\gamma=\{\gamma_{X}:A^*\otimes\mathscr C(X,X)\longrightarrow A\}_{X\in Ob(\mathscr C)}$. Then, a morphism $\phi:\mathcal M\longrightarrow\mathcal M'$  in $Com_A^{\mathscr C}(\psi)$ has a section (resp. a retraction) in $Com_A^{\mathscr C}(\psi)$ if and only if it has a section (resp. a retraction) in $Com^{\mathscr C}$.
\end{thm}

We conclude with the following result for the category $Com_A^{\mathscr C}(\psi)$ of entwined comodules, which can be verified in a manner similar to Theorem \ref{T7.5}.

\begin{Thm}\label{T7.7}
Let $(\mathscr C,A,\psi)$ be an entwining structure. Suppose there exists a normalized cointegral  $\gamma=\{\gamma_{X}:A^*\otimes\mathscr C(X,X)\longrightarrow A\}_{X\in Ob(\mathscr C)}$  on $(\mathscr C,A,\psi)$. Then, we have

\smallskip
(a) The  functor $G_\psi: Com_A^{\mathscr C}(\psi)\longrightarrow Com^{\mathscr C}$ is a semisimple functor.

\smallskip
(b) The  functor $G_\psi:  Com_A^{\mathscr C}(\psi)\longrightarrow Com^{\mathscr C}$ is a Maschke functor.

\end{Thm}

\end{document}